%% file: main.tex
\documentclass[reqno]{amsproc}
\usepackage[utf8]{inputenc}
\usepackage{amsfonts}
\usepackage{graphicx}
\usepackage{amscd}
\usepackage{url}
\usepackage{amsmath}
\usepackage{amsthm}
\usepackage{amssymb}
\usepackage{latexsym}
\usepackage{hyperref}
\usepackage{xcolor}
\usepackage[all]{xy}
\usepackage{color}
\usepackage{mathrsfs}
\makeatletter
\patchcmd\@thm{\let\thm@indent\indent}{\let\thm@indent\noindent}%
  {}{}
\makeatother
\theoremstyle{plain}
\newtheorem{theorem}{\bf Theorem}
\newtheorem{lemma}{\bf Lemma}
\newtheorem{corollary}{\bf Corollary}
\newtheorem{proposition}{\bf Proposition}
\DeclareMathOperator{\sech}{sech}

\theoremstyle{definition}
\newtheorem{remark}{\bf Remark}
\newtheorem{example}{\bf Example}

\input{pre/pkgs} 
\input{pre/cmds} 

\begin{document}
\title[Gradient Einstein-type warped metrics]{Gradient Einstein-type warped products: rigidity, existence and nonexistence results via a nonlinear PDE}

\author[José N. V. Gomes]{José Nazareno Vieira Gomes$^1$}
\author[Willian I. Tokura]{Willian Isao Tokura$^2$}
\address{$^{1}$Departamento de Matemática, Centro de Ciências Exatas e Tecnologia, Universidade Federal de São Carlos, São Paulo, Brazil.}
\address{$^2$Faculdade de Ciências Exatas e Tecnologia, Universidade Federal da Grande Dourados, Mato Grosso do Sul, Brazil.}
\email{$^1$jnvgomes@ufscar.br}
\email{$^2$williantokura@ufgd.edu.br}
\urladdr{$^1$https://www2.ufscar.br}
\urladdr{$^2$https://portal.ufgd.edu.br}

\maketitle
{\centering\footnotesize Dedicated to the memory of Romildo Pina, our beloved friend and great mathematician.\par}
\begin{abstract}
We establish the necessary and sufficient conditions for constructing gradient Einstein-type warped metrics. One of these conditions leads us to a general Lichnerowicz equation with analytic and geometric coefficients for this class of metrics on the space of warping functions. In this way, we prove gradient estimates for positive solutions of a nonlinear elliptic differential equation on a complete Riemannian manifold with associated Bakry–Émery Ricci tensor bounded from below. As an application, we provide nonexistence and rigidity results for a large class of gradient Einstein-type warped metrics. Furthermore, we show how to construct gradient Einstein-type warped metrics, and then we give explicit examples which are not only meaningful in their own right, but also help to justify our results.
\end{abstract}

\section{Introduction and statements of the main results}\label{intro}

A geodesically complete Riemannian metric $g$ on a $d$-dimensional smooth manifold $M^d$ is a \textit{gradient Einstein-type metric} if there is a smooth function $h$ on $M^d$ such that
\begin{equation}\label{fundamental}
	\alpha \riccitensor+\beta\hessian{h}+\mu \tensorproduct{dh}{dh}=(\rho\scalarcurvature_{g}+\lambda)g
\end{equation}
for certain constants $\alpha$, $\beta$, $\mu$ and $\rho$ with $(\alpha,\beta,\mu)\neq(0,0,0)$ and some smooth function $\lambda$ on $M^d$. Here, $h$ is the \textit{potential function},  $dh$ is the 1-form metrically dual to the gradient vector field $\nabla h$, the tensor $\hessian{h}$ is the Hessian of $h$, $\riccitensor$ is the Ricci tensor of $g$, and $\scalarcurvature_{g}=\textup{tr(Ric)}$ is the scalar curvature. Notice that, this definition, which was introduced by Catino, Mastrolia, Monticelli and Rigoli~\cite{catino2016geometry}, unifies various particular cases that have been well-studied in the literature about self-similar solutions of geometric flows. 

A gradient Einstein-type metric is \textit{nondegenerate} if $\beta^{2}\neq(d-2)\alpha\mu$ and $\beta\neq0$. The case of $\beta^{2}=(d-2)\alpha\mu$ and $\beta\neq0$ refers to a \textit{degenerate} gradient Einstein-type metric. This terminology is due to the fact that a degenerate gradient Einstein-type metric is equivalent to a conformally Einstein metric (see Remark~\ref{Char-CE}). The $\beta=0$ case was addressed separately in~\cite{catino2016geometry}. Shortly after, the first author provided a different approach for the nondegenerate case, from which we know the most appropriate setting to work on the class of proper, noncompact, nondegenerate, gradient Einstein-type metrics with constant Ricci curvatures (see Gomes~\cite[Eq.~(3.10) and Theorem~3]{gomes2019note}). 

In this article, we study gradient Einstein-type metrics that are realized as warped metrics. We have been motivated by the well-known fact that the warped metrics appear naturally in Riemannian geometry and enjoy abundant applications. Not only the literature about this subject is already very rich, but also, it is not unlikely, that the latter may play a fundamental role in the understanding of countless physical facts. 

We use the term {\it gradient Einstein-type warped metric} to describe the warped metric $g=g_B + f^2 g_F$ on the product manifold $M^d = B^n \times F^m$ when it satisfies equation \eqref{fundamental} for some warping function $f$ on $B^n$. For simplicity, and since there is no danger of confusion with the Bishop and O'Neill's notation, we refer to $B^n\times_f F^m$ as a {\it gradient Einstein-type warped product}. The Riemannian manifolds $(B^{n},g_B)$ and $(F^{m},g_F)$ are called the base and the fiber of the warped product, respectively. If $f$ is a constant, then $(M^d,g)$ is the standard Riemannian product. For greater detail, the reader may wish to consult Bishop and O'Neill~\cite{bishop1969}.

Our first theorem provides the necessary and sufficient conditions for constructing gradient Einstein-type warped products $B^n\times_f F^m$. For our purposes, we assume that the potential function of a gradient Einstein-type warped product is the lift $\tilde{h}$ of a smooth function $h$ on the base to the product. The reader will be able to observe that this condition holds for a large class of gradient Einstein-type warped products (see Proposition~\ref{thm:wpgEtM-and-lfB}). Besides, throughout this paper, we work with $\beta\neq 0$, which includes both degenerate and nondegenerate cases. 
\begin{theorem}\label{thm:gEtM-c}
The manifold $B^{n}\times_{f}F^{m}$ is a gradient Einstein-type
  warped product with $\beta\neq{0}$, potential function
  $\tilde{h}$ and soliton function $\tilde{\lambda}$
  if and only if
\begin{enumerate}[(a)]
\item \label{2a}
$\alpha\riccitensor<g_{B}>+\beta\hessian{h}+\mu\tensorproduct{dh}{dh}=\dfrac{\alpha{m}}{f}\hessian{f}+\left(\rho\scalarcurvature_{g}+\lambda\right)g_{B}$ \;\mbox{on}\; $(B^n,g_B)$, \;\hbox{and}
\item \label{2b} The scalar curvature $R_{g_{F}}$ of $(F^{m},g_{F})$ is a constant satisfying
\begin{align*}
\left(\dfrac{\alpha}{m}\!-\!\rho\right)R_{g_{F}}\!=\!\rho R_{g_B}f^{2}\!+\!\lambda f^{2}\!+\!(\alpha\!-\!2m\rho)f\laplacian{f}\!-\!\beta f\gradient{h}(f)\!+\!(m\!-\!1)(\alpha\!-\!m\rho)\norm{\gradient{f}}^{2}.
\end{align*}
\hspace{-0,6cm}Moreover, the fiber must be an Einstein manifold for the case $\alpha\neq 0$.
  \end{enumerate}
\end{theorem}

The importance of Theorem~\ref{thm:gEtM-c} is twofold. Firstly, it gives all conditions for (both) the existence and the construction of gradient Einstein-type warped metrics. For instance, by combining it with the theory of Lie symmetry groups for Partial Differential Equation (PDE), we present a method to construct explicit examples of degenerate and nondegenerate cases of such metrics (see Section~\ref{HCEtWP}). Secondly, Theorem~\ref{thm:gEtM-c} provides the means to investigate nonexistence results of gradient Einstein-type warped products. Indeed, from item~(b) we notice that the warping function $f$ of a gradient Einstein-type warped product $B^n\times_f F^m$ necessarily provides the existence of a positive solution $u=f^{1/\sigma(m)}$ of 
\begin{equation}\label{eq:lichnerowicz-type-equation}
  \sigma(m)\driftedlaplacian{u}{w} +\frac{\rho\scalarcurvature<g_{B}>+\lambda}{\alpha-2m\rho}u-
  \frac{(\alpha-m\rho)R_{g_{F}}}{m(\alpha-2m\rho)}u^{1-2\sigma(m)}=0
\end{equation}
on $C^{\infty}(B)$, where
\begin{equation*}
    \sigma(m)=\dfrac{\alpha-2m\rho}{m[\alpha-\rho(1+m)]}, \quad\alpha\notin\{2m\rho, (1+m)\rho\} \quad \hbox{and} \quad
    w=\dfrac{\beta h}{\alpha-2m\rho}.
\end{equation*}
Here,
$\driftedlaplacian{}{w}{}$ is the
drifted Laplacian, which is defined by 
\[\driftedlaplacian{}{w}{u}:= \laplacian u-\langle\gradient w,\gradient u\rangle=e^{w}\divergent{e^{-w}\gradient{}{u}},\]
where $\laplacian{}$ stands for the Laplace–Beltrami operator on $(B^n,g_B)$. The novelty of equation~\eqref{eq:lichnerowicz-type-equation} comes from the fact that the existence of gradient Einstein-type warped products depends on solutions of this PDE with nonlinear term of exponent $1-2\sigma(m)$ being able to assume any real value. The more general version of \eqref{eq:lichnerowicz-type-equation} (see~\eqref{PDE-geral} below) with a nonlinear term of exponent greater than one has been treated more frequently due to its correlation with the Yamabe problem. The case where this exponent is less than one has not been extensively studied due to its lack of applications.


As we mentioned previously, \eqref{eq:lichnerowicz-type-equation} is a simplified version of the nonlinear PDE
\begin{equation}\label{PDE-geral}
  \driftedlaplacian{u(x)}{\varphi}+\mathcal{A}(x)u(x)+\mathcal{B}(x)u(x)^\varepsilon=0, \quad\mbox{where}\quad
  \mathcal{A}, \mathcal{B}\in\functionspace{B}.
\end{equation}
It is well-known that the existence of solutions to \eqref{PDE-geral} is related to its coefficients $\mathcal{A}$ and $\mathcal {B}$ as well as the value of $\varepsilon$. For instance, when $\varphi=\text{constant}$, $\mathcal{A}\equiv0$ and $0<\varepsilon<(n+2)/(n-2)$, Gidas and Spruck \cite{gidas1981global} proved that any nonnegative solution to \eqref{PDE-geral} is identically zero provided the Ricci tensor of the Riemannian metric is nonnegative. In the case $\varphi=\text{constant}$ and $\varepsilon>1$, it is equivalent to the Yamabe problem on noncompact Riemannian manifolds, see, e.g., Brandolini, Rimoldi and Setti~\cite{brandolini1998positive} or Yamabe~\cite{yamabe1960deformation}. Equation~\eqref{PDE-geral} is commonly known as the Fisher-KPP equation when the parameters are $\varepsilon = 2$, $\mathcal{A} = c$ and $\mathcal{B} = -c$, where $c$ is a positive constant. When the parameters are $\varepsilon = 3$, $\mathcal{A} = 1$ and $\mathcal{B} = -1$, it is widely recognized as the Allen–Cahn equation. The Allen–Cahn equation originally emerged from studying phase separation processes in iron alloys, including order-disorder transitions. Moreover, this equation has connections to the exploration of minimal surfaces, making it an intriguing topic within the field of differential geometry, see Del Pino, Kowalczyk and Wei \cite{del2013entire}. In addition, \eqref{PDE-geral} is closely related to the Lichnerowicz equation coming from Hamiltonian's constraint equation for the Einstein-scalar field system in general relativity, see~\cite{Guglielmo2016,Song2010} and the references therein. 

From now on it is clear that the most appropriate setting to study existence or not of positive solutions of \eqref{eq:lichnerowicz-type-equation} is by means of the geometry of each factor of the product and the parameters of the gradient Einstein-type warped metric. More specifically, it depends on the relative signs and asymptotic behavior of the coefficients 
\[
\sigma(m)\in\reals, \qquad \quad \frac{\rho\scalarcurvature<g_{B}>+\lambda}{\alpha-2m\rho}\in C^\infty(B)\qquad\text{and}\qquad \frac{(\alpha-m\rho) R_{g_{F}}}{m(\alpha-2m\rho)}\in \mathbb{R}.
\]
The well-known tool to work in this setting is by means of gradient estimates, since it plays a fundamental role in comprehending solutions of elliptic differential equations similar to~\eqref{PDE-geral}. For instance, Li and Yau~\cite{li1986parabolic} established a gradient estimate for harmonic functions using the maximum principle on Riemannian manifolds whose Ricci curvature is bounded from below. This estimate played a pivotal role in the formulation of the Liouville theorem. In 1993, Hamilton 
\cite{hamilton1993matrix} derived an elliptic-type gradient estimate for positive solutions of the heat equation on compact manifolds, which became known as the Hamilton-type gradient estimate. In 2006, Souplet and Zhang \cite{souplet2006sharp} made significant advancements by developing a localized version of the Hamilton-type gradient estimate. This accomplishment was realized by combining Li-Yau's Harnack inequality and Hamilton's gradient estimate. Further advancements in the field were made by Wu \cite{wu2019gradient}, who provided local elliptic gradient estimates, encompassing both Hamilton's type and Souplet-Zhang's type for positive solutions on smooth metric measure spaces.

Our second theorem is a local gradient estimate for positive solutions of~\eqref{PDE-geral} on a complete Riemannian manifold with the Bakry–Émery Ricci tensor
\[\textup{Ric}^{\varphi}:=\riccitensor+\hessian{\varphi}\]
bounded from below.
\begin{theorem}\label{Th2}
Let $(B^{n},g)$ be an $n$-dimensional complete Riemannian manifold. Assume that $\textup{Ric}^{\varphi}\geqslant-(n-1)\textup{K}$, for some $\textup{K}\geqslant0$, in the geodesic ball $B(x_{0},R)$ centered at $x_{0}\in B^n$ of radius $R\geqslant2$. Consider a smooth positive solution $u(x)$ of
\begin{equation}\label{PDE-Th2}
  \driftedlaplacian{u(x)}{\varphi}+\mathcal{A}(x)u(x)+\mathcal{B}u(x)^{\varepsilon}=0, \quad \mathcal{A}\in\functionspace{B},\quad \mathcal{B}\in\reals,
\end{equation}
in $B(x_{0},R)$, such that $u(x)\leqslant D$, for some constant $D$. There exists a constant $C$ such that
\begin{equation*}
\begin{split}
    \norm{\gradient{\ln u}}&\leqslant C\left(q-p\ln u\right)\Bigg{\{}\dfrac{1}{R}+\sqrt{\textup{K}}+\dfrac{\sqrt{\gamma_{\driftedlaplacian{}{\varphi}{}}}}{\sqrt{R}}+\sup_{B(x_{0},R)}\norm{\gradient{\mathcal{A}}}^{\frac{1}{3}}+\sqrt{\mathcal{A}^{+}}\\
    &\quad+\sup_{B(x_{0},R)}\sqrt{\left[\left(\varepsilon-1+\dfrac{p}{q-p\ln u}\right)\mathcal{B}\right]^{+}}\sup_{B(x_{0},R)}u^{\frac{\varepsilon-1}{2}}\Bigg{\}}
\end{split}
\end{equation*}
in $B(x_{0},R/2)$, where $q$ and $p>0$ are constants chosen so that $q-p\ln u\geqslant\delta>0$ for some constant $\delta$. Here $\mathcal{A}^{+}(x):=\max \{\mathcal{A}(x),0\}$ and
$\gamma_{\driftedlaplacian{}{\varphi}{}}:= \max_{\partial B(x_{0},1)}\driftedlaplacian{}{\varphi}{r},$
where $r(x)$ is the distance function from $x_0$.
\end{theorem}
\begin{remark} Theorem~\ref{Th2} extends the technique used in~\cite{hamilton1993matrix,li1986parabolic,souplet2006sharp} to the Bakry–Émery Ricci tensor and its associated drifted Laplacian. Taking $\varphi=cte$, $p=1$, $q=1+\ln(D)$, $\delta=1$ and $\mathcal{A}=\mathcal{B}=0$ in Theorem~\ref{Th2}, we recover Theorem~1.1 of~\cite{souplet2006sharp} for harmonic functions. For $\mathcal{A}$ and $\mathcal{B}$ nonzero, Wu~\cite[Theorem~1.1]{wu2019gradient} provides a similar result to Theorem~\ref{Th2} with parameters $p=1$, $q=1+\ln(D)$, $\delta=1$ and $\varepsilon\in\mathbb{R}$. In this setting, Wu proved some rigidity results when $\varepsilon\geqslant1$ and $\mathcal{B}$ is nonconstant. This latter case has been treated more frequently due to its correlation with the Yamabe problem. The case $\varepsilon<1$ and constant $\mathcal{B}$ were not addressed in~\cite{wu2019gradient}. The novelty of Theorem~\ref{Th2} comes from the possibility for treating the case $\varepsilon\in\mathbb{R}$ together with $\mathcal{B}$ constant. Note that $\varepsilon$ being arbitrary is crucial for us because, depending on the parameters of the gradient Einstein-type manifold, the nonlinear exponent $\varepsilon$ can assume any real value.
\end{remark}

The first application of Theorem~\ref{Th2} is a rigidity result for the class of gradient Einstein-type standard Riemannian products $B^n\times F^m$ with 
\[
\textup{Ric}_{g_B}^{w}\geqslant 0, \quad\mbox{where}\quad w=\dfrac{\beta{h}}{\alpha-2m\rho}.
\]
This is the content of the following corollary.

\begin{corollary}\label{Cor1-triviality}
Let $B^n\times_{f} F^m$ be a gradient Einstein-type warped product satisfying equation~\eqref{eq:lichnerowicz-type-equation} on a noncompact base $B^n$ with $\textup{Ric}_{g_B}^{w}\geqslant 0$ and coefficients satisfying either 
\begin{enumerate}[(a)]
\item\label{3a} $\sigma(m)>0$, \ $\dfrac{\rho\scalarcurvature<g_{B}>+\lambda}{\alpha-2m\rho}<0$ $(=0)$ \   and \ $\dfrac{R_{g_F}(\alpha-m\rho)}{\alpha-2m\rho}<0$ $(=0)$,\quad\hbox{or}
\item\label{3b} $\sigma(m)<0$, \ $\dfrac{\rho\scalarcurvature<g_{B}>+\lambda}{\alpha-2m\rho}>0$ $(=0)$  \ and \ $\dfrac{R_{g_F}(\alpha-m\rho)}{\alpha-2m\rho}>0$ $(=0)$.
\end{enumerate}
Then, $B^n\times_{f} F^m$ must be a gradient Einstein-type standard Riemannian product provided 
$$\sup_{B(x_{0}, R)}
\|\rho\gradient{\scalarcurvature<g_{B}>+\gradient{\lambda}}\| =o(R^{-\frac{3}{2}})\; as \; R\to+\infty$$
and the warping function satisfying $f(x)=e^{o(r^{\frac{1}{2}}(x))}$ near infinity.
\end{corollary}

\begin{remark}
The assumptions on the sign of the constant $R_{g_F}(\alpha-m\rho)/(\alpha-2m\rho)$ in Corollary~\ref{Cor1-triviality} are necessary, otherwise, such a metric cannot exist (see Corollary~\ref{Cor2-nonexistence}). In Section~\ref{HCEtWP}, we construct an example of a gradient Einstein-type warped product with a nonconstant warping function which satisfies the conditions of Corollary~\ref{Cor1-triviality} except for the assumption on the warping function (see Example~\ref{EX1}). Hence, this example shows that the growth condition for $f$ in this corollary cannot be dropped. In this section we also show how to construct examples of gradient Einstein-type standard Riemannian products.
\end{remark}

Under the same setup of Corollary~\ref{Cor1-triviality}, if we assume that $R_{g_F}(\alpha-m\rho)/(\alpha-2m\rho)$ is either positive in item~\eqref{3a} or negative in item~\eqref{3b}, then we obtain a nonexistence result for gradient Einstein-type warped metrics since~\eqref{eq:lichnerowicz-type-equation} has no positive solution in these cases. This nonexistence result is the content of Corollary~\ref{Cor2-nonexistence} of Section~\ref{sec:gradient-estimates}, which is the second application of Theorem~\ref{Th2}. In Section~\ref{Further-discussions}, we discuss particular cases of gradient Einstein-type metrics related to Corollaries~\ref{Cor1-triviality} and \ref{Cor2-nonexistence}. We apply these corollaries to well-known solitons in the literature, such as Ricci solitons, $\rho$-Einstein solitons and Einstein manifolds. 

Nonexistence results of gradient Einstein-type manifolds were also addressed in \cite{catino2016geometry}. For it, the authors observed that the potential function $h$ of any degenerate gradient Einstein-type manifold $(M^{k},g)$ necessarily provides the existence of a positive solution $u=e^{-\frac{k-2}{2}\frac{\mu}{\beta}h}$ of the following Yamabe equation
\begin{equation*}
4 \frac{k-1}{k-2} \Delta u-R_{g} u+\Lambda u^{\frac{k+2}{k-2}}=0
\end{equation*}
for some $\Lambda \in \mathbb{R}$, while the potential function $h$ of any nondegenerate gradient Einstein-type manifold $(M^k, g)$ necessarily provides the existence of a positive solution $u=e^{\frac{\mu}{\beta} h}$ of PDE
\begin{equation*}
\Delta u-\frac{\mu}{\beta}[k \lambda+(k \rho-\alpha) R_{g}]u=0.
\end{equation*}
In both cases, the authors used spectral theory to show when an Einstein-type manifold cannot exist, see~\cite[Section~3]{catino2016geometry} or  Bianchini, Mari and Rigoli~\cite{bianchini2015yamabe}  for details. Therefore, our approach is quite different from theirs.

\section{How to construct gradient Einstein-type warped products?}\label{HCEtWP}

In this section, we show how to construct gradient Einstein-type warped products. The first step is to analyze how the potential function behaves on a gradient Einstein-type warped product.
For instance, Borges and Tenenblat~\cite{Borges} proved that the potential function of any gradient Ricci soliton warped product with a nonconstant warping function depends only on the base (see also Lemma~6 in~\cite{gomes2021note}). The first author, in joint work with Agila~\cite[Theorem 3]{agila2022geometrical}, showed that any complete gradient $\rho$-Einstein soliton warped product $B^n\times_fF^m$ with a nonconstant warping function has fiber of constant scalar curvature and the potential function depends only on the base if either one of the following three conditions holds: $m\neq2$, or $\rho\neq-1/6$, or $f$ is bounded. We observe that Theorem~3 of \cite{agila2022geometrical} also proves that the potential function of any complete gradient Yamabe soliton warped product with a nonconstant warping function depends only on the base. Here, we prove the following general result. 

\begin{proposition}\label{thm:wpgEtM-and-lfB}
  Let $B^{n}\times_{f}F^{m}$ be a gradient Einstein-type warped product with $\beta\neq{0}$ and nonconstant warping function. The following cases hold:
 \begin{enumerate}[(a)]
    \item  \label{1a}
      If $\mu=0$, then the potential function $h$ is the lift to the product of a smooth function on $B^{n}$ if and only if
      $\rho{\scalarcurvature<g_{F}>}+\lambda{f}^{2}$ is the lift of a smooth
      function on $B^{n}$ to the product;
    \item \label{1b}
      If $\mu\neq{0}$ and the potential function $h$ is the lift to the product of a smooth function on $B^{n}$, then
      $\rho{\scalarcurvature<g_{F}>}+\lambda{f}^{2}$ is also the lift to the product of a smooth function on $B^{n}$.
  \end{enumerate}
\end{proposition}

\begin{remark}
The converse statement of item~\eqref{1b} is not true. Indeed,
 by a straightforward computation, we can prove that $\reals\times_{\cosh(t)}\reals$ is a gradient Einstein-type warped product satisfying
  \begin{equation*}
    \riccitensor+\hessian{h}+\tensorproduct{dh}{dh}=\left(\frac{R_{g}}{2}+1\right)g
  \end{equation*}
with potential function
\[
    h(t,s)=\ln2+\ln\cosh(t)+\ln\cosh(s).
\] 
  Note that the potential function $h$ does not depend only on the base.
\end{remark}

\begin{proof}[Proof of Proposition~\ref{thm:wpgEtM-and-lfB}] 
\noindent{Item~\eqref{1a}}: Take $X\in\mathcal{L}(B)$ and $U\in\mathcal{L}(F)$ in~\eqref{fundamental} to get
\begin{equation*}
0=\hessian{h(X,U)}{}=X(U(h))-(\nabla_{X}U)h= X(U(h))-\frac{X(f)U(h)}{f}=f X(U(hf^{-1})).
\end{equation*}
This shows that $U(hf^{-1})$ depends only on the fiber, which implies $h=\varphi+f\psi$ where $\varphi\in\functionspace{B}$ and $\psi\in\functionspace{F}$. Now, take a unitary geodesic $\gamma$ on $B^{n}$, so that equation~\eqref{fundamental} reads along $\gamma$ as
\begin{align*}
\varphi^{\prime\prime}+f^{\prime\prime}\psi
& = \gamma^{\prime}\left(\gamma^{\prime}(h)\right)                                                            
= \hessian{h}\left(\gamma^{\prime},\gamma^{\prime}\right)                                                   \\
& =\frac{\rho}{\beta}\left(\scalarcurvature<g_{B}>+\frac{\scalarcurvature<g_{F}>}{f^{2}}-\frac{2m}{f}\driftedlaplacian{f}{}-m(m-1)\frac{\|\nabla f\|^{2}}{f^{2}}\right)+\frac{\lambda}{\beta}\\
& \quad -\frac{\alpha}{\beta}\left(\textup{Ric}_{g_B}(\gamma',\gamma')-\frac{m}{f}\hessian{f(\gamma',\gamma')}\right).
\end{align*}
Therefore $\beta f^{2}f^{\prime\prime}U(\psi)=U\left(\rho\scalarcurvature<g_{F}>+f^{2}\lambda\right)$ for every $U\in\mathcal{L}(F)$.
Since $f$ is nonconstant, there exists $p\in B^{n}$ such that $f^{\prime\prime}(p)\neq{0}$. Thus, $U(\psi)=0$ if and only if $U\left(\rho\scalarcurvature<g_{F}>+f^{2}\lambda\right)=0$, i.e., $h$ is the lift of a smooth function on $B^{n}$ to $B^{n}\times{F}^{m}$ if and only if $\rho\scalarcurvature<g_{F}>+f^{2}\lambda$ is the lift of a smooth function on $B^{n}$ to $B^{n}\times{F}^{m}$.
 
\vspace{0.2cm}
\noindent{Item~\eqref{1b}}: We can make the change of variable $v=e^{\frac{\mu}{\beta}h}$ so that equation~\eqref{fundamental} becomes
\begin{equation*}
\alpha \riccitensor+\frac{\beta^{2}}{\mu}\frac{\hessian{v}}{v}=(\rho R_{g}+\lambda)g.
\end{equation*}
Similarly to item~\eqref{1a}, we deduce that $v=\varphi+f\psi$ where $\varphi\in\functionspace{B}$ and $\psi\in\functionspace{F}$. Therefore, for every geodesic $\gamma$ in $B^{n}$ we get
\begin{align*}
\frac{\beta}{\mu}\frac{1}{v}\left(\varphi^{\prime\prime}+f^{\prime\prime}\psi \right)
& =  \frac{\beta}{\mu}\frac{1}{v}\hessian{v}\left(\gamma^{\prime},\gamma^{\prime}\right)\\
& =\frac{\rho}{\beta}\left( \scalarcurvature<g_{B}>+\frac{\scalarcurvature<g_{F}>}{f^{2}}-\frac{2m}{f}\driftedlaplacian{f}{}-m(m-1)\frac{\|\nabla f\|^{2}}{f^{2}}\right)+\frac{\lambda}{\beta}\\
& \quad -\frac{\alpha}{\beta}\left(Ric_{B}(\gamma',\gamma')-\frac{m}{f}\hessian{f(\gamma',\gamma')}\right).
\end{align*}
Since $h$ depends only on the base, we have $U(\psi)=0$ for every $U\in\mathcal{L}(F)$. Thus, $\rho\scalarcurvature<g_{F}>+f^{2}\lambda$ is the lift of a smooth function on $B^{n}$ to $B^{n}\times{F}^{m}$.
\end{proof}

Proposition~\ref{thm:wpgEtM-and-lfB} establishes the most appropriate context to work with gradient Einstein-type warped metric. It was crucial for considering the main parameters setting on Theorem~\ref{thm:gEtM-c}, which is the second step to proceed with our construction. Before proving this theorem, we will use it as a tool to show a way for constructing examples of gradient Einstein-type warped products.

Having completed our preliminary discussion, we are now ready to construct our examples. 
\subsection{Construction of gradient Einstein-type warped product} In this subsection, we construct examples of gradient Einstein-type warped metrics by reducing a PDE to an ODE using a method known as \textit{ansatz}. An important method for generating ansatz is based on the theory of Lie symmetry groups for PDE (see Olver's book \cite{olver2000applications} or Bluman and Kumei's book \cite{bluman1989symmetries}). For instance, using a rotational ansatz, Robert Byrant showed the existence of a complete rotationally symmetric steady gradient Ricci soliton on $\mathbb{R}^{n}$, $n\geqslant3$, which is unique up to homothety (see Chow et al.~\cite{chow2007ricci}).

We start by taking $(B^n,g_{B})\!=\!(\mathbb{R}^n,\Psi^{-2} g_{Euc})$, $n\geqslant2$,  and an Einstein manifold $(F^{m},g_{F})$ satisfying $\textup{Ric}_{g_F}=\theta g_{F}$. For an arbitrary choice of a nonzero vector $\bar{\alpha}=\left(\alpha_1, \ldots, \alpha_n\right)$ in $\mathbb{R}^{n}$, we define the function $\xi: \mathbb{R}^n \rightarrow \mathbb{R}$ by
$$
\xi\left(x_1, \ldots, x_n\right)=\alpha_1 x_1+\ldots+\alpha_n x_n.
$$
Next, we look for smooth functions of one variable $\Psi, f, h,\lambda:(a, b) \subseteq \mathbb{R} \rightarrow \mathbb{R}$ with $\Psi$ and $f$ being positive and such that 
\begin{equation*}\Psi=\Psi \circ \xi, \ \ f=f \circ \xi, 
 \ \ h=h \circ \xi, \ \ \lambda=\lambda \circ \xi\ : \ \xi^{-1}(a, b)\subseteq \mathbb{R}^n \rightarrow \mathbb{R}
\end{equation*}
satisfy items~\eqref{2a} and~\eqref{2b} of Theorem~\ref{thm:gEtM-c} on $\mathbb{R}^n\times F^{m}$ with warped metric
$$
g=\Psi(x)^{-2}g_{Euc}+f(x)^2 g_F.
$$

First, note that the Ricci tensor on the conformal metric $g_B=\Psi^{-2} g_{Euc}$ is given by:
\begin{equation}\label{ricci_base}
\begin{split}
(\riccitensor_{g_B})_{ij}&=\frac{1}{\Psi^2}\left\{(n-2) \Psi(\hessian\Psi)_{ij}+\left[\Psi\Delta\Psi-(n-1)\norm{\nabla\Psi}^2\right] \delta_{ij}\right\}\\
&=\frac{1}{\Psi^2}\left\{(n-2)\alpha_{i}\alpha_{j} \Psi\Psi^{\prime \prime}+\left[\Psi\|\bar{\alpha}\|^2\Psi^{\prime \prime} -(n-1)\|\bar{\alpha}\|^2\left(\Psi^{\prime}\right)^2\right] \delta_{ij}\right\}.
\end{split}
\end{equation}

Hence, we easily see that the scalar curvature on the conformal metric takes the following form
\begin{equation}\label{scalar_base}
\begin{aligned}
R_{g_B}=\sum_{k=1}^n \Psi^2\left(\riccitensor_{g_B}\right)_{k k}=\|\bar{\alpha}\|^2(n-1)\left[2 \Psi \Psi^{\prime \prime}-n\left(\Psi^{\prime}\right)^2\right].
\end{aligned}
\end{equation}
Now, in order to compute the Hessian $\hessian{h}$ relatively to $g_B$ we evoke the expression
$$
(\hessian{h})_{i j}=h_{x_i x_j}-\sum_{k=1}^n \Gamma_{i j}^k h_{x_k},
$$
where $h_{x_{k}}$ and $h_{x_{k}x_{k}}$ denotes the first and second order derivative of $h$, respectively. In this case, the Christoffel symbols $\Gamma_{i j}^k$ for distinct $i, j, k$ are given by
$$
\Gamma_{i j}^k=0,\quad \Gamma_{i j}^i=-\frac{\Psi_{x_j}}{\Psi},\quad \Gamma_{i i}^k= \frac{\Psi_{x_k}}{\Psi} \quad \mbox{and}\quad \Gamma_{i i}^i=-\frac{\Psi_{x_i}}{\Psi} .
$$
Therefore,
\begin{equation}\label{hes}
\begin{aligned}
(\hessian{h})_{i j} & =h_{x_i x_j}+\Psi^{-1}\left(\Psi_{ x_i} h_{x_j}+\Psi_{x_j} h_{x_i}\right)-\delta_{i j} \sum_k \Psi^{-1} \Psi_{x_k} h_{x_k} \\
& =\alpha_i \alpha_j h^{\prime \prime}+\left(2 \alpha_i \alpha_j-\delta_{i j} \|\bar{\alpha}\|^2\right)\Psi^{-1} \Psi^{\prime} h^{\prime}.
\end{aligned}
\end{equation}
Using \eqref{hes} we have the following expression for $\Delta f$ with respect to $g_B$
\begin{equation}\label{laplaciano}
   \Delta f= \sum_k \Psi^2 (\hessian{f})_{k k}=\|\bar{\alpha}\|^2 \Psi^2\left[f^{\prime \prime}-(n-2) \Psi^{-1} \Psi^{\prime} f^{\prime}\right].
\end{equation}
Now, the expression of $\langle\nabla f, \nabla h\rangle$ and $\|\nabla f\|^2$ on conformal metric $g_B$ are given by
\begin{equation}\label{finais}
\begin{aligned}\langle\nabla f, \nabla h\rangle&=\Psi^2 \sum_k f_{x_k} h_{ x_k}=\|\bar{\alpha}\|^2 \Psi^2 f^{\prime} h^{\prime},\\
\quad\|\nabla f\|^2&=\Psi^2 \sum_k  f_{x_k}^2=\|\bar{\alpha}\|^2 \Psi^2\left(f^{\prime}\right)^2.
\end{aligned}
\end{equation}
From equations \eqref{scalar_base}, \eqref{laplaciano} and \eqref{finais}, we have the following expression for the scalar curvature of $g$:
\begin{equation}\label{scalar_curvature}
\begin{split}
R_{g}&=R_{g_B}-2m\frac{\Delta f}{f}-m(m-1) \frac{\norm{\nabla f}^2}{f^2}\\
&=\|\bar{\alpha}\|^2\bigg{\{}(n-1)\left[2 \Psi \Psi^{\prime \prime}-n\left(\Psi^{\prime}\right)^2\right]-2m \Psi^2\left[\dfrac{f^{\prime \prime}}{f}-(n-2) \frac{\Psi^{\prime}}{\Psi} \dfrac{f^{\prime}}{f}\right]\\
&\quad-m(m-1) \Psi^2\left(\frac{f^{\prime}}{f}\right)^2\bigg{\}}.
\end{split}
\end{equation}

Note that equations~\eqref{ricci_base} and~\eqref{hes} split into two cases since $n\geqslant2$. So, combining \eqref{ricci_base}-\eqref{scalar_curvature} with Theorem~\ref{thm:gEtM-c}, we get the following result. 

\begin{proposition}\label{Proposition-invariant}
    With $\Psi=\Psi \circ \xi$, $f=f \circ \xi$,  $h=h \circ \xi$, $\lambda=\lambda \circ \xi$ and $\xi$ as above, the manifold $\xi^{-1}(a,b)\times F^{m}$ with warped metric
$$
g=\Psi(x)^{-2}g_{Euc}+f(x)^2 g_F
$$
is a gradient Einstein-type manifold if and only if $\Psi$, $f$, $h$ and $\lambda$ satisfy
\begin{equation}\label{Prop2-1}
 \alpha(n-2) \frac{\Psi^{\prime \prime}}{\Psi}-\alpha m\frac{f^{\prime \prime}}{f} -2 \alpha m\frac{f^{\prime}}{f} \frac{\Psi^{\prime}}{\Psi} +\beta h^{\prime\prime}+2\beta\frac{\Psi^{\prime}}{\Psi}h^{\prime}+\mu(h^{\prime})^{2}=0,
\end{equation}
\begin{equation}\label{Prop2-2}
    \begin{aligned}
\alpha \Psi \Psi^{\prime \prime}-&\alpha(n-1)\left(\Psi^{\prime}\right)^2+\alpha m \Psi \Psi^{\prime}\frac{f^{\prime}}{f}-\beta\Psi \Psi^{\prime} h^{\prime}\\
& =\rho\bigg{\{}(n-1)\left[2 \Psi \Psi^{\prime \prime}-n\left(\Psi^{\prime}\right)^2\right]-2m \Psi^2\left[\dfrac{f^{\prime \prime}}{f}-(n-2) \frac{\Psi^{\prime}}{\Psi} \dfrac{f^{\prime}}{f}\right]\\
&\quad-m(m-1) \Psi^2\left(\frac{f^{\prime}}{f}\right)^2\bigg{\}}+\dfrac{\lambda}{\|\bar{\alpha}\|^2}
\end{aligned}
\end{equation}
and
\begin{equation}\label{Prop2-3}
    \begin{aligned}
\frac{\alpha \theta}{\|\bar{\alpha}\|^{2} f^{2}}&+\Psi^2 \bigg{\{}-\alpha \left[\dfrac{f^{\prime \prime}}{f}-(n-2) \frac{\Psi^{\prime}}{\Psi}  \dfrac{f^{\prime}}{f}\right]-\alpha(m-1) \left(\dfrac{f^{\prime}}{f}\right)^2+\beta h^{\prime} \frac{f^{\prime}}{f} \bigg{\}}\\
& =\rho\bigg{\{}(n-1)\left[2 \Psi \Psi^{\prime \prime}-n\left(\Psi^{\prime}\right)^2\right]-2m \Psi^2\left[\dfrac{f^{\prime \prime}}{f}-(n-2) \frac{\Psi^{\prime}}{\Psi}\dfrac{f^{\prime}}{f}\right]\\
&\quad-m(m-1) \Psi^2\left(\frac{f^{\prime}}{f}\right)^2\bigg{\}}+\dfrac{\lambda}{\|\bar{\alpha}\|^2}.
\end{aligned}
\end{equation}
\end{proposition}

It is worth mentioning that the parameters in equations~\eqref{Prop2-1}, \eqref{Prop2-2} and \eqref{Prop2-3} are the functions $f$, $h$, $\Psi$ and $\lambda$, while the system consists of three equations. Therefore, we can fix one of these functions and solve the system to determine the other functions. For instance, to obtain a gradient Einstein-type standard Riemannian product it is enough to get $f$ to be a constant in these equations. Also, whenever the system above has a solution for a nonzero vector $\bar{\alpha}=\left(\alpha_1, \ldots, \alpha_n\right)\in\mathbb{R}^{n}$ with $\|\bar{\alpha}\|^{2}=a\in\reals$, then it has infinitely many examples since we can choose a different $\bar{\alpha}\in\mathbb{R}^{n}$ with $\|\bar{\alpha}\|^{2}=a\in\reals$. Each of these examples has one associated geometry.

As applications of Proposition~\ref{Proposition-invariant}, we now construct some examples of degenerate and nondegenerate gradient Einstein-type warped metrics. 

\begin{example}\label{EX1} Let $\reals^{n}$ be the Euclidean space with standard metric $g_{Euc}$ and $\xi=\sum_{i=1}^{n}\alpha_{i}x_{i}$ with $\bar{\alpha}=(\alpha_{1},\dots, \alpha_{n})\in\reals^{n}$ and $\|\bar{\alpha}\|^2=1$. If we take a complete Ricci flat manifold $(F^{m},g_{F})$, then the functions
\[\Psi(\xi)\equiv 1, \qquad f(\xi)=e^{\xi} \qquad h(\xi)=\sqrt{m}\xi \qquad\mbox{and}\qquad \lambda(\xi)\equiv-m^{2}-m\]
satisfy equations \eqref{Prop2-1}, \eqref{Prop2-2}, \eqref{Prop2-3} and define a gradient Einstein-type warped metric on $\reals^{n}\times F^{m}$ with warping function $f$, potential function $\tilde{h}$ and parameters 
\[\alpha=1, \qquad \beta=\sqrt{m}, \qquad \mu=1  \quad\mbox{and}\quad \rho=-1.\]
Note that this example is nondegenerate if $n\neq2$, and degenerate otherwise.
\end{example}
\begin{example}\label{EX2}Let us consider the hyperbolic space $\mathbb{H}^{n}$, namely, the open half space $\mathbb{R}_{+}^{n}=
\{
(x_{1},  \dots, x_{n})\in\mathbb{R}^{n}: x_{n}>0
\}$ with the metric $x_{n}^{-2}g_{Euc}$ and $(F^{m},g_{F})$ a complete Ricci flat manifold. Besides, consider $\alpha_{1}=\cdots=\alpha_{n-1}=0$ and $\alpha_{n}=1$. Then the functions
\begin{align*}
\Psi(x_{n})&=x_{n},\quad
f(x_{n})=\frac{1}{x_{n}}, \quad
h(x_{n})=\ln(x_{n}+1)-\ln(x_{n})\quad \hbox{and}\\
\lambda(x_{n})&=1-\alpha(n+m-1)-\dfrac{x_{n}}{x_{n}+1}+\rho(n+m)(n+m-1)
\end{align*}
satisfy equations \eqref{Prop2-1}, \eqref{Prop2-2}, \eqref{Prop2-3}. Hence, $\mathbb{H}^{n}\times_{f}F^{m}$ is a gradient Einstein-type warped product with warped metric 
\[g=g_{\mathbb{H}^{n}}+x_{n}^{-2}g_{F},\] and parameters 
\[\alpha\in\mathbb{R}-\{0\},\qquad \beta=1,\qquad \mu=1 \quad\mbox{and}\quad\rho\in\mathbb{R}.\]
Note that this example is nondegenerate if $\alpha\neq \frac{1}{n+m-2}$, and degenerate otherwise.
\end{example}

Since the base and the fiber in the previous examples are both complete, we conclude from Lemma 7.2 of Bishop and O'Neill \cite{bishop1969} that Examples \ref{EX1} and \ref{EX2} are complete.
\begin{example}\label{EX4}
Let $\mathbb{R}_{+}^{n}=\{(x_{1}, \dots, x_{n})\in\mathbb{R}^{n}: x_{n}>0\}$ be the half space with the metric $\coth(x_{n})^{2}g_{Euc}$ and $(F^{m},g_{F})$ a complete Ricci flat manifold. Besides, consider $\alpha_{1}=\cdots=\alpha_{n-1}=0$ and $\alpha_{n}=1$. Then, the functions
\begin{align*}
\Psi(x_{n})&=\tanh(x_{n}),\;\;
f(x_{n})=\coth(x_{n}), \;\;
h(x_{n})=\frac{2\alpha}{3\beta}(m+n-2)\ln[\cosh(x_{n})]\;\; \hbox{and}\\
\lambda(x_{n})&= -\dfrac{\alpha}{3 \cosh^{4}(x_{n})}\big{[}2(n+m-2)+(n+m+1)\cosh(2x_{n})\big{]}-\rho R_{g}(x_{n})
\end{align*}
satisfy equations \eqref{Prop2-1}, \eqref{Prop2-2}, \eqref{Prop2-3}. Thus, $\mathbb{R}_{+}^{n}\times_{f}F^{m}$ is a nondegenerate gradient Einstein-type warped metric with 
\begin{equation}\label{ggff}
g=\coth(x_{n})^{2}\big(g_{Euc}+g_{F}\big),
\end{equation}
and parameters 
\[\alpha\in\reals-\{0\},\qquad \beta\in\reals-\{0\},\qquad \mu=0 \quad\mbox{and}\quad \rho\in\reals.\] 
To show that $g$ is complete, we use the divergent curve criterion. Since the fiber $(F^{m},g_{F})$ is complete, it is enough to prove that $(\mathbb{R}_{+}^{n}, \coth^2(x_{n})g_{Euc})$ is complete. 

Suppose that $\gamma:[0,\infty)\rightarrow \mathbb{R}_{+}^{n}$, $\gamma(t)=(\gamma_{1}(t), \dots, \gamma_{n}(t))$ is a divergent curve. Then
\begin{eqnarray*}
\int^t_{0}\coth(\gamma_{n}(t))|\gamma^{\prime}(t)|dt &\geqslant& \int^t_{0}\coth(\gamma_{n}(t))|\gamma_{n}^{\prime}(t)|dt\\
&\geqslant&\bigg{|} \int^t_{0}\coth(\gamma_{n}(t))\gamma_{n}^{\prime}(t)dt\bigg{|}\\
&=& \Big{|}\ln(\sinh(\gamma_{n}(t)))-\ln(\sinh(\gamma_{n}(0)))\Big{|}.
\end{eqnarray*}
If either $\gamma_{n}(t)\rightarrow+\infty$ or $\gamma_{n}(t)\rightarrow0$ as $t\rightarrow+\infty$, we have
\begin{equation*}
\lim_{t\rightarrow+\infty}\int_{0}^t \coth(\gamma_{n}(t))|\gamma^{\prime}(t)|dt=+\infty.
\end{equation*}
On the other hand, if $\gamma_{n}(t)$ is not part of the previous cases, then for some $j$, we have $\gamma_{j}(t)\rightarrow+\infty$ as $t\rightarrow+\infty$. In this case, we get
\begin{eqnarray*}
\int^t_{0}\coth(\gamma_{n}(t))|\gamma^{\prime}(t)|dt &\geqslant&\int^t_{0}\sqrt{\gamma_{1}^{\prime}(t)^{2}+\cdots+\gamma_{n}^{\prime}(t)^{2}}dt\\
&\geqslant& \int^t_{0}|\gamma_{j}^{\prime}(t)|dt\\
&\geqslant& \big{|}\gamma_{j}(t)-\gamma_{j}(0)\big{|}.
\end{eqnarray*}
Hence,
\begin{equation*}
\lim_{t\rightarrow+\infty}\int_{0}^t \coth(\gamma_{n}(t))|\gamma^{\prime}(t)|dt=+\infty.
\end{equation*}
So, we obtain the required completeness of $(\mathbb{R}_{+}^{n}\times_f F^m, g)$, where $g$ is given by \eqref{ggff}. 
\end{example}

\begin{example}\label{EX5}Consider the Euclidean space $\mathbb{R}^{n}$ with the metric $\cosh(x_{n})^{2}g_{Euc}$ and $(F^{m},g_{F})$ a complete Ricci flat manifold. Besides, consider $\alpha_{1}=\cdots=\alpha_{n-1}=0$ and $\alpha_{n}=1$. Then the functions
\begin{align*}
\Psi(x_{n})&=\sech(x_{n}),\;\;
f(x_{n}) =\cosh(x_{n}), \;\;
h(x_{n})=\frac{\alpha}{\beta}(m+n-2)\ln[\cosh(x_{n})]\;\; \hbox{and}\\
\lambda(x_{n})&=\frac{\rho}{2}(m+n-1) \operatorname{sech}^4(x_{n})\Big{\{}(m+n-2)[\cosh (2 x_{n})-1]+4\Big{\}}-\alpha \operatorname{sech}^4(x_{n})
\end{align*}
satisfy equations \eqref{Prop2-1}, \eqref{Prop2-2}, \eqref{Prop2-3}. Therefore, $\mathbb{R}^{n}\times_{f}F^{m}$ is a degenerate gradient Einstein-type warped metric with 
\begin{equation*}
g=\cosh(x_{n})^{2}\big(g_{Euc}+g_{F}\big),
\end{equation*}
and parameters 
\[\alpha\in\reals-\{0\},\qquad \beta\in\reals-\{0\},\qquad \mu=\frac{\beta^{2}}{\alpha(m+n-2)} \quad\mbox{and}\quad\rho\in\reals.\] 

To show that $g$ is complete, it is sufficient to prove that $(\mathbb{R}^{n}, \cosh(x_{n})^{2}g_{Euc})$ is complete. Since $\cosh(t)\geqslant1$, $\forall t\in \mathbb{R}$, we can conclude that  
\begin{equation}\label{cauchy}d_{\cosh^{2}(x_{n})g_{Euc}}(p,q)\geqslant d_{g_{Euc}}(p,q),\qquad \forall\ p,q\in\mathbb{R}^{n},
\end{equation}
where $d_{\cosh^{2}(x_{n})g_{Euc}}$ and $d_{g_{Euc}}$ are the distance functions induced by $\cosh^{2}(x_{n})g_{Euc}$ and $g_{Euc}$, respectively. Equation \eqref{cauchy} implies that any Cauchy sequence on the metric space $(\mathbb{R}^{n}, d_{\cosh^{2}(x_{n})g_{Euc}})$ is a Cauchy sequence on the metric space $(\mathbb{R}^{n}, d_{Euc})$. Since $(\mathbb{R}^{n}, g_{Euc})$ is complete, we deduce from the equivalences of the Hopf-Rinow theorem that $(\mathbb{R}^{n}, \cosh(x_{n})^{2}g_{Euc})$ is complete.
\end{example}

\begin{remark}\label{Char-CE}
We recall the following characterization that justifies the terminology used in \cite{catino2016geometry}. A Riemannian manifold $(M, g)$ is conformally Einstein if and only if for some nonzero $\alpha, \beta$ and $\mu$, $(M, g)$ is a degenerate gradient Einstein-type manifold. Example~\ref{EX4} shows that $\mu\neq0$ condition is crucial for this characterization, while Example~\ref{EX5} aligns with it.
\end{remark}

Now, we present two examples of metrics that satisfy conditions \eqref{Prop2-1}, \eqref{Prop2-2}, \eqref{Prop2-3}, meeting all the conditions to be a gradient Einstein-type warped metric except for the completeness.

\begin{example}\label{Incomplete1} Let $\reals_{\ast}^{n}=\{(x_{1},\dots,x_{n})\in\mathbb{R}^{n}:\xi=\alpha_{1}x_{1}+\dots+\alpha_{n}x_{n}>0\}$ equipped with the Euclidean metric $g_{Euc}$ and $(F^{m}, g_{F})$ a complete Einstein manifold with $\textup{Ric}_{g_{F}}=(m-2) g_{F}$, $m>2$. Let us consider 
  \[\alpha=1, \qquad \beta=-1, \qquad \mu=1  \quad\mbox{and}\quad \rho=0.  \]
Then the functions
\[\Psi(\xi)\equiv 1, \qquad f(\xi)=\xi \qquad h(\xi)=-\ln(\xi) \qquad\mbox{and}\qquad \lambda(\xi)\equiv0\]
satisfy the conditions for constructing a gradient Einstein-type warped product, except for the completeness of the manifold. The metric $g_{Euc}$ on $\reals_{\ast}^{n}$ is not complete, since the divergent curve $\gamma:[0,\infty)\rightarrow(\frac{p_{1}}{1+t},\dots,\frac{p_{n}}{1+t})$ on $\mathbb{R}_{\ast}^{n}$ satisfies
\begin{eqnarray*}
\lim_{t\rightarrow+\infty}\int_1^t \|\gamma^{\prime}(t)\|dt=p_{1}^{2}+\dots+p_{n}^{2}<+\infty,
\end{eqnarray*}
where $(p_{1},\dots, p_{n})\in\mathbb{R}_{\ast}^{n}$. 
\end{example}

From Proposition~\ref{Proposition-invariant} we also recover the following example of~\cite[Corollary 1]{feitosa2019gradient}.

\begin{example}Consider $(\reals^{n},e^{2\xi}g_{Euc})$ and $(F^{m},g_{F})$ a complete Ricci flat manifold. Let 
\[\alpha=1, \qquad \beta=1, \qquad \mu=0  \quad\mbox{and}\quad \rho=0.  \]
Then the functions
\[\Psi(\xi)=e^{-\xi}, \quad f(\xi)=e^{\xi},\quad h(\xi)=-\dfrac{2-n-m}{2}\xi\quad\mbox{and}\quad \lambda(\xi)=\dfrac{2-n-m}{2}e^{-2\xi}\]
satisfy the conditions for constructing a gradient Einstein-type warped product, except for the completeness of the manifold. To show that the warped product $\mathbb{R}^{n}\times_{f}F^{m}$ is not complete, we can use a similar argument as in Example \ref{Incomplete1}.
\end{example}

\section{Proof of Theorem~\ref{thm:gEtM-c}}\label{Theorem1}

To prove Theorem~\ref{thm:gEtM-c} is crucial that we state the following two lemmas. Lemma~\ref{prop:einstein-type-manifold-property-beta-neq-0} provides an extension of equation (2.4) of Proposition 2.1 in \cite{eminenti2008ricci}. Lemma~\ref{Lemma2-Hamilton} is an application of Lemma~\ref{prop:einstein-type-manifold-property-beta-neq-0} for the particular case of gradient Einstein-type warped metrics. 

\begin{lemma}\label{prop:einstein-type-manifold-property-beta-neq-0}
Let $(M^{n},g)$ be a gradient Einstein-type manifold with $\beta\neq{0}$. Then
\begin{align*}
&d\big{[}\alpha^{2}\scalarcurvature+\beta^{2}\norm{\gradient{h}}^{2}-2(n-1)\alpha\big(\rho\scalarcurvature+\lambda\big)\big{]}\\
&= 2\mu\big(\alpha\laplacian{h}{}-\beta\norm{\gradient{h}}^{2}\big)dh-\alpha\mu{d}\norm{\gradient{h}}^{2}+2\beta(\rho\scalarcurvature+\lambda)dh.
\end{align*}
\end{lemma}
\begin{proof}Throughout this proof, we will utilize the following identities, which are well-known in the literature:
\begin{equation}\label{eqqq1}
    \divergent{\tensorproduct{df}{df}}=\laplacian{f}df+\hessian{f}(\gradient{f}),
\end{equation}
\begin{equation}\label{eqqq2}
\divergent{\hessian{f}}=\riccitensor(\gradient{f})+d\laplacian{f}
\end{equation}
and
\begin{equation}\label{eqqq3}
d\norm{\gradient{f}}^{2}=2\nabla^2 f(\nabla f, \cdot)
\end{equation}
as well as the following equations derived from the fundamental equation \eqref{fundamental}:
\begin{equation}\label{eqqq4}\alpha\riccitensor(\gradient{h})+\beta \hessian{h}(\gradient{h})+\mu\tensorproduct{dh}{dh}(\gradient{h},\cdot)=(\rho\scalarcurvature+\lambda)g(\gradient{h},\cdot),
\end{equation}
\begin{equation}\label{eqqq5}\alpha\scalarcurvature+\beta\laplacian{h}{}+\mu\norm{\gradient{h}}^{2}=n(\rho\scalarcurvature+\lambda)
\end{equation}
and 
\begin{equation}\label{eqqq6}\beta d\laplacian{h}{}=n d(\rho\scalarcurvature+\lambda)-\alpha d\scalarcurvature-\mu d\norm{\gradient{h}}^{2}.
\end{equation}
From the second contracted Bianchi identity and equations~\eqref{eqqq1},~\eqref{eqqq2},~\eqref{eqqq6}, we have
\begin{align*}
\alpha{d\scalarcurvature} &= 2\divergent{\alpha\riccitensor}\\
&= -2\beta\divergent{\hessian{h}}-2\mu\divergent{\tensorproduct{dh}{dh}}+2d(\rho\scalarcurvature+\lambda)\\
&=-2\beta\divergent{\hessian{h}}-2\mu\laplacian{h}{}dh-2\mu\hessian{h}(\gradient{h})+2d(\rho\scalarcurvature+\lambda)\\
&=-2\beta\riccitensor(\gradient{h})-2\beta d\laplacian{h}{}-2\mu\laplacian{h}{}dh-2\mu\hessian{h}(\gradient{h})+2d(\rho\scalarcurvature+\lambda)\\
&=-2\beta\riccitensor(\gradient{h})-2nd(\rho\scalarcurvature+\lambda)+2\alpha d\scalarcurvature+2\mu{d}\norm{\gradient{h}}^{2}-2\mu\laplacian{h}{}dh\\
&\quad -2\mu\hessian{h}(\gradient{h})+2d(\rho\scalarcurvature+\lambda).
\end{align*}
Therefore, from~\eqref{eqqq3} and \eqref{eqqq4}, we arrive at
\begin{align*}
\alpha{d\scalarcurvature} &= 2(n-1)d(\rho\scalarcurvature+\lambda)+2\beta \riccitensor(\gradient{h})-2\mu\hessian{h}(\gradient{h})+2\mu\laplacian{h}{}dh\\
&=\frac{2(\beta^{2}+\alpha\mu)}{\beta}\riccitensor(\gradient{h})+\frac{2 dh}{\beta}[(n-1)(\rho\scalarcurvature+\lambda)\mu-\alpha\mu \scalarcurvature]\\
&\quad+2(n-1)d(\rho\scalarcurvature+\lambda).
\end{align*}
Now, utilizing~\eqref{eqqq3}, \eqref{eqqq4} and \eqref{eqqq5}, we compute
\begin{align*}
&d\big{[}\alpha^{2}\scalarcurvature+\beta^{2}\norm{\gradient{h}}^{2}-2(n-1)\alpha(\rho\scalarcurvature+\lambda)\big{]}\\
&=\frac{2\alpha(\beta^{2}+\alpha\mu)}{\beta}\riccitensor(\gradient{h})+\frac{2\alpha dh}{\beta}[(n-1)(\rho\scalarcurvature+\lambda)\mu-\alpha\mu \scalarcurvature] \\
		 &\quad +2(n-1)\alpha d(\rho\scalarcurvature+\lambda)+2\beta^{2}\hessian{h}(\gradient{h})-2(n-1)\alpha d(\rho\scalarcurvature+\lambda)\\
&= 2\alpha\beta \riccitensor(\gradient{h})+\frac{2\alpha^{2}\mu}{\beta}\riccitensor(\gradient{h})+2\beta^{2}\hessian{h}(\gradient{h})               \\
		 &\quad+\frac{2\alpha\mu}{\beta}dh[(n-1)(\rho\scalarcurvature+\lambda)-\alpha \scalarcurvature]\\
          &=2\beta(\rho\scalarcurvature+\lambda)dh-2\beta\mu\norm{\gradient{h}}^{2}dh+\frac{2\alpha^{2}\mu}{\beta}\riccitensor(\gradient{h})          \\
          & \quad+\frac{2\alpha\mu}{\beta}dh[n(\rho\scalarcurvature+\lambda)-\alpha \scalarcurvature]-\frac{2\alpha\mu}{\beta}(\rho\scalarcurvature+\lambda) dh \\
          &=2\beta(\rho\scalarcurvature+\lambda)dh-2\beta\mu\norm{\gradient{h}}^{2}dh+\frac{2\alpha^{2}\mu}{\beta}\riccitensor(\gradient{h})           \\
          &\quad +\frac{2\alpha\mu}{\beta}dh(\beta\laplacian{h}{}+\mu\norm{\gradient{h}}^{2})-\frac{2\alpha\mu}{\beta}(\rho\scalarcurvature+\lambda) dh\\
          &=2\mu\big(\alpha\laplacian{h}{}-\beta\norm{\gradient{h}}^{2}\big)dh-\alpha\mu{d}\norm{\gradient{h}}^{2}+2\beta(\rho\scalarcurvature+\lambda)dh,
	\end{align*}
from which we complete the proof of the lemma.
\end{proof}

\begin{lemma}\label{Lemma2-Hamilton}
  Let $B^{n}\times_{f}F^{m}$ be a gradient Einstein-type manifold with
  $\beta\neq{0}$, potential function $\tilde{h}$ and soliton function
  $\tilde{\lambda}$. Then
\begin{align*}
&d\big{(}\alpha(2-n-m)(\rho\scalarcurvature+\lambda)-\alpha\beta\laplacian{h}{}-\frac{\alpha\beta m}{f}\gradient{h}(f)+\left(\beta^{2}-\alpha\mu\right)\norm{\gradient{h}}^{2}\big{)}\\
&=-\alpha\mu d\norm{\gradient{h}}^{2}+2\beta(\rho\scalarcurvature+\lambda) dh + 2\mu\left(\alpha\laplacian{h}{}+\frac{\alpha{m}}{f}\gradient{h}(f)-\beta\norm{\gradient{h}}^{2}\right)d h.
\end{align*}
\end{lemma}
\begin{proof}
Combining $\alpha \scalarcurvature=n(\rho\scalarcurvature+\tilde{\lambda})-\beta\laplacian{\tilde{h}}{}-\mu\norm{\gradient{\tilde{h}}}^{2}$ with Lemma~\ref{prop:einstein-type-manifold-property-beta-neq-0}, we obtain
\begin{align*}
&d\big{(}\alpha(2-n-m)(\rho\scalarcurvature+\tilde{\lambda})-\alpha\beta\laplacian{\tilde{h}}{}+\left(\beta^{2}-\alpha\mu\right)\norm{\gradient{\tilde{h}}}^{2}\big{)}\\
&=2\mu\left(\alpha\laplacian{\tilde{h}}{}-\beta\norm{\gradient{\tilde{h}}}^{2}\right)d\tilde{h}-\alpha\mu d\norm{\gradient{\tilde{h}}}^{2}+2\beta(\rho\scalarcurvature+\tilde{\lambda}) d\tilde{h}.
\end{align*}
Since $h$ only depends on $B^n$, one has
\[\gradient{\tilde{h}}= \gradient{h},\qquad\laplacian{\tilde{h}}{}=\laplacian{h}{}+\frac{m}{f}\gradient{h}(f).\]
The result follows by substituting these identities in the previous equation. 
\end{proof}

Now, we are in position to prove Theorem~\ref{thm:gEtM-c}.

\begin{proof}[Proof of Theorem~\ref{thm:gEtM-c}] We will proof the case $\alpha\neq0$, since the case $\alpha=0$ is analogous.
Let $(B^{n}\times_{f}F^{m},g)$ be a warped product manifold with $g=g_B+f^2g_F$. The warped metric $g$ on $B^{n}\times F^{m}$ defines a gradient Einstein-type metric with potential function $\tilde{h}$ and parameters $\alpha$, $\beta$, $\mu$, $\rho$ and $\tilde{\lambda}$ if and only if the following tensorial equation is satisfied
\begin{equation}\label{thm-1:eq1}
	\alpha \textup{Ric}+\beta\nabla d\tilde{h}+\mu \tensorproduct{d\tilde{h}}{d\tilde{h}}=(\rho\scalarcurvature+\tilde{\lambda})g
\end{equation}
for any vector fields $X_{1}$, $Y_{1}\in \mathcal{L}(B)$ and $X_{2}$, $Y_{2}\in \mathcal{L}(F)$. 

For $X_{1}\in \mathcal{L}(B)$ and $Y_{1}\in \mathcal{L}(B)$, we use Lemma~34 and Corollary 43 of~\cite{o1983semi} to deduce that
\begin{equation*}
\riccitensor\left(X_1, Y_1\right)=\riccitensor_{g_B}\left(X_1, Y_1\right)-\frac{m}{f} \hessian{f} \left(X_1, Y_1\right), \quad \nabla d\tilde{h} (X_{1},Y_{1})=\widetilde{\nabla dh(X_{1},Y_{1})}.
\end{equation*}
Hence, in this case, equation~\eqref{thm-1:eq1} is equivalent to 
\begin{equation}\label{thm-1:eq2}
\alpha \riccitensor<g_{B}>+\beta \hessian{h}+\mu \tensorproduct{dh}{dh}=\frac{\alpha{m}}{f}\hessian{f}+\left(\rho\scalarcurvature+\lambda\right)g_{B},
\end{equation}
which proves Item~\eqref{2a}.
 
For $X_{1}\in \mathcal{L}(B)$ and $Y_{2}\in \mathcal{L}(F)$, we use again Lemma~34 and Corollary 43 of~\cite{o1983semi} to deduce that
\begin{equation*}
\riccitensor\left(X_1, Y_2\right)=0,\quad \nabla d\tilde{h}(X_{1},Y_{2})=0,\quad d\tilde{h}(Y_{2})=0.
\end{equation*}
So, equation~\eqref{thm-1:eq1} is trivially verified. Now, since 
$$\hessian{\tilde{h}(X_{2},Y_{2})}=f\gradient{h}(f)g_{F}{}(X_{2},Y_{2}),\qquad X_{2},Y_{2}\in\mathcal{L}(F),$$
we deduce from Corollary 43 of~\cite{o1983semi} and~\eqref{thm-1:eq1} that $(F^{m},g_{F})$ satisfies
$\riccitensor<g_{F}>=\theta{g_{F}}$ with
\begin{equation}\label{thm-1:eq3}
\theta=\frac{R_{g_{F}}}{m}=\alpha^{-1}\left(\rho\scalarcurvature+\lambda\right)f^{2}+f\laplacian{f}+(m-1)\norm{\gradient{f}}^{2}-\alpha^{-1}\beta f\gradient{h}(f).
\end{equation}
To conclude the proof, we will prove that $\theta$ is constant. By tracing~\eqref{thm-1:eq2} and calculating its derivative, we obtain
\begin{equation}\label{thm-1:eq4}\alpha d\scalarcurvature<g_{B}>=n d(\rho\scalarcurvature+\lambda)-\alpha{m}\frac{df}{f^{2}}\laplacian{f}+\alpha{m}\frac{d\laplacian{f}}{f}-\beta d\laplacian{h}{}-\mu d\norm{\gradient{h}}^{2}.
\end{equation}
From \eqref{thm-1:eq2} it yields
\begin{equation*}
\begin{split}
\alpha \divergent{\riccitensor<g_{B}>}=& d(\rho\scalarcurvature+\lambda)+\alpha{m} \divergent{\frac{\hessian{f}}{f}}-\beta \divergent{\hessian{h}}\\
& -\mu \divergent{\tensorproduct{dh}{dh}}\\
=& d(\rho\scalarcurvature+\lambda)+\alpha{m}\left[\frac{\divergent{\hessian{f}}}{f}-\frac{\hessian{f}(\gradient{f})}{f^{2}}\right]-\beta \divergent{\hessian{h}}\\
& -\mu \divergent{\tensorproduct{dh}{dh}}\\
=& d(\rho\scalarcurvature+\lambda)+\frac{\alpha{m}}{f}(\riccitensor(\gradient{f})+d\laplacian{f})-\frac{\alpha{m}}{2f^{2}}d\norm{\gradient{f}}^{2}-\beta \riccitensor(\gradient{h})\\
&-\beta d\laplacian{h}{}-\mu \divergent{\tensorproduct{dh}{dh}}\\
=& d(\rho\scalarcurvature+\lambda)+\frac{\alpha{m}}{f}(\riccitensor(\gradient{f})+d\laplacian{f})-\frac{\alpha{m}}{2f^{2}}d\norm{\gradient{f}}^{2}-\beta \riccitensor(\gradient{h})\\
&-\beta d\laplacian{h}{}-\mu\laplacian{h}{} dh-\frac{\mu}{2}d\norm{\gradient{h}}^{2}.
\end{split}
\end{equation*}
Replacing the expressions
\begin{align*}
\alpha \riccitensor(\gradient{f}) & =(\rho\scalarcurvature+\lambda) df+\frac{\alpha{m}}{2f}d\norm{\gradient{f}}^{2}-\beta\hessian{h}(\gradient f)-\mu\gradient{h}(f)dh,\\
\beta \riccitensor(\gradient{h})  & =\frac{\beta(\rho\scalarcurvature+\lambda)}{\alpha} dh+\frac{\beta m}{f}\hessian{f}(\gradient{h})-\frac{\beta^{2}}{2\alpha}d\norm{\gradient{h}}^{2}-\frac{\mu\beta}{\alpha}\norm{\gradient{h}}^{2}dh,
\end{align*}
in the previous relation, we immediately deduce 
\begin{equation}\label{thm-1:eq5}
\begin{split}
\alpha \divergent{\riccitensor<g_{B}>} & =d(\rho\scalarcurvature+\lambda)+\frac{m}{f}(\rho\scalarcurvature+\lambda) df+\frac{\alpha{m}^{2}}{2f^{2}}d\norm{\gradient{f}}^{2}-\frac{m\beta}{f}\hessian{h}(\nabla f)\\
&\quad-\frac{m\mu}{f}\gradient{h}(f)dh+\frac{\alpha{m}}{f}d\laplacian{f}-\frac{\alpha{m}}{2f^{2}}d\norm{\gradient{f}}^{2}-\frac{\beta(\rho\scalarcurvature+\lambda)}{\alpha} dh\\
&\quad-\frac{ m\beta}{f}\hessian{f(\gradient{h},\cdot)}{}+\frac{\beta^{2}}{2\alpha}d\norm{\gradient{h}}^{2}+\frac{\mu\beta}{\alpha}\norm{\gradient{h}}^{2}dh-\beta d\laplacian{h}{}\\
&\quad-\mu\laplacian{h}{} dh-\mu \hessian{h}(\gradient{h})\\
& =d(\rho\scalarcurvature+\lambda)+\frac{m}{f}(\rho\scalarcurvature+\lambda) df+\frac{\alpha{m}^{2}}{2f^{2}}d\norm{\gradient{f}}^{2}-\frac{m\beta}{f}d(\gradient{h}(f))\\
&\quad -\frac{m\mu}{f}\gradient{h}(f)dh+\frac{\alpha{m}}{f}d\laplacian{f} -\frac{\alpha{m}}{2f^{2}}d\norm{\gradient{f}}^{2}-\frac{\beta(\rho\scalarcurvature+\lambda)}{\alpha}dh\\
&\quad+\frac{\beta^{2}}{2\alpha}d\norm{\gradient{h}}^{2}+\frac{\mu\beta}{\alpha}\norm{\gradient{h}}^{2}dh-\beta d\laplacian{h}{}-\mu\laplacian{h}{} dh-\frac{\mu}{2}d\norm{\gradient{h}}^{2} \\
&=d(\rho\scalarcurvature+\lambda)+\frac{m}{f}(\rho\scalarcurvature+\lambda) df+\frac{\alpha{m}(m-1)}{2f^{2}}d\norm{\gradient{f}}^{2}\\
&\quad -\frac{m\beta}{f}d(\gradient{h}(f))-\frac{m\mu}{f}\gradient{h}(f)dh+\frac{\alpha{m}}{f}d\laplacian{f}-\frac{\beta(\rho\scalarcurvature+\lambda)}{\alpha} dh\\
&\quad+\left(\frac{\beta^{2}}{2\alpha}-\frac{\mu}{2}\right)d\norm{\gradient{h}}^{2}+\frac{\mu\beta}{\alpha}\norm{\gradient{h}}^{2}dh-\beta d\laplacian{h}{}-\mu\laplacian{h}{} dh,
\end{split}
\end{equation}
where we have used the following identity
\[d(\gradient{h}(f))=\hessian{h}(\gradient f)+\hessian{f}(\gradient{h}).\]
Now, substituting~\eqref{thm-1:eq4} and \eqref{thm-1:eq5} in the twice-contracted Bianchi identity, we have
\begin{equation}\label{thm-1:eq6}
\begin{split}
0&=-\frac{\alpha d \scalarcurvature<g_{B}>}{2}+\alpha \divergent{\riccitensor<g_{B}>}\\
&=-\frac{n}{2} d(\rho\scalarcurvature+\lambda)+\frac{\alpha{m}}{2}\frac{df}{f^{2}}\laplacian{f}-\frac{\alpha{m}}{2}\frac{d\laplacian{f}}{f}+\frac{\beta}{2} d\laplacian{h}{}+\frac{\mu}{2} d\norm{\gradient{h}}^{2}\\
&\quad+d(\rho\scalarcurvature+\lambda)+\frac{m}{f}(\rho\scalarcurvature+\lambda) df+\frac{\alpha{m}(m-1)}{2f^{2}}d\norm{\gradient{f}}^{2}\\
&\quad -\frac{m\beta}{f}d(\gradient{h}(f))-\frac{m\mu}{f}\gradient{h}(f)dh+\frac{\alpha{m}}{f}d\laplacian{f}-\frac{\beta(\rho\scalarcurvature+\lambda)}{\alpha} dh\\
&\quad+\left(\frac{\beta^{2}}{2\alpha}-\frac{\mu}{2}\right)d\norm{\gradient{h}}^{2}+\frac{\mu\beta}{\alpha}\norm{\gradient{h}}^{2}dh-\beta d\laplacian{h}{}
-\mu\laplacian{h}{} dh.
\end{split}
\end{equation}
On the other hand, from Lemma~\ref{Lemma2-Hamilton} we can write
\begin{align*}
\frac{\beta(\rho\scalarcurvature+\lambda)}{\alpha} d h&=\frac{(2-n-m)}{2}d(\rho\scalarcurvature+\lambda)-\frac{\beta}{2} d\laplacian{h}{}-\frac{\beta m}{2} d\left(\frac{\gradient{h}(f)}{f}\right)+\frac{\mu}{2} d\norm{\gradient{h}}^{2}\\
&\quad+\frac{1}{2}\left(\frac{\beta^{2}}{\alpha}-\mu\right)d\norm{\gradient{h}}^{2}-\mu\big{(}\laplacian{h}{}+\frac{m}{f}\gradient{h}(f)-\frac{\beta}{\alpha}\norm{\gradient{h}}^{2}\big{)}d h.
\end{align*}
So, substituting the above expression in~\eqref{thm-1:eq6}, we obtain
\begin{equation*}
\begin{split}
0&=\frac{m}{2}d(\rho\scalarcurvature+\lambda)+\frac{\alpha{m}}{2}\frac{\laplacian{f}}{f^{2}}df+\frac{\alpha{m}}{2}\frac{d\laplacian{f}}{f}+m(\rho\scalarcurvature+\lambda)\frac{df}{f}+\frac{\alpha{m}(m-1)}{2f^{2}}d\norm{\gradient{f}}^{2}\\
&\quad-\frac{m\beta}{2f}d(\gradient{h}(f))-\frac{m\beta}{2f^{2}}\gradient{h}(f)df.
\end{split}
\end{equation*}
Multiplying above equation by $\frac{2f^{2}}{m}$, we get
\begin{equation*}
\begin{split}
\alpha d\theta=&f^{2}d(\rho\scalarcurvature+\lambda)+\alpha\laplacian{f} df+\alpha f d\laplacian{f}+2(\rho\scalarcurvature+\lambda) f df+\alpha(m-1)d\norm{\gradient{f}}^{2}\\
&-\beta f d(\gradient{h}(f))-\beta\gradient{h}(f)df=0,
\end{split}
\end{equation*}
which proves that $\theta$ is constant. To conclude Item~\eqref{2b}, we only need to replace the scalar curvature
\begin{equation*}
R_{g}=R_{g_B}+\frac{R_{g_F}}{f^2}-2m\frac{\Delta f}{f}-m(m-1) \frac{\norm{\nabla f}^2}{f^2}.
\end{equation*}
into equation \eqref{thm-1:eq3}.
\end{proof}

\section{Proof of the gradient estimate and its applications}\label{sec:gradient-estimates}

In this session, we give the proof of the gradient estimates mentioned in our introduction. As applications, we prove Corollary~\ref{Cor1-triviality} and discuss some estimates concerning the nonexistence of gradient Einstein-type warped metrics.

\begin{proof}[Proof of Theorem~\ref{Th2}]Let $p\in (0,\infty)$ and consider the change $h=\ln u^{p}=p\ln u$ in \eqref{PDE-Th2}. Then we have
\begin{equation}\label{thm-2:eq1}
    \driftedlaplacian{}{\varphi}{h}+\dfrac{1}{p}\norm{\gradient{h}}^{2}+p\mathcal{A}+p\mathcal{B}e^{\frac{h}{p}(\varepsilon-1)}=0.
\end{equation}
Suppose that there exists a constant $q$ such that $q-h\geqslant\delta>0$ for some $\delta$. Then, for all $x\in B(x_{0},R)$ the function 
\begin{equation}\label{thm-2:eq2}
    G:=\norm{\gradient{\ln (q-h)}}^{2}=\dfrac{\norm{\gradient{h}}^{2}}{(q-h)^{2}}
\end{equation}
satisfies
\begin{equation}\label{thm-2:eq3}
    \begin{split}
\dfrac{1}{2}\driftedlaplacian{}{\varphi}{G}&\geqslant\dfrac{(p-q+h)}{p(q-h)}\langle \gradient{G},\gradient{h}\rangle-(n-1)\textup{K}G+(q-h)G^{2}\\
&\quad -\dfrac{p}{(q-h)^{2}}\langle\gradient{h},\gradient{\mathcal{A}}\rangle-\dfrac{p\mathcal{A}}{q-h}G-\left(\varepsilon-1+\dfrac{p}{q-h}\right)\mathcal{B}e^{\frac{h}{p}(\varepsilon-1)}G.
    \end{split}
\end{equation}
Indeed, working in a local orthonormal system and using the Einstein summation convention, we compute 
\begin{equation*}
    \begin{split}
        G_{j}&=\dfrac{2h_{i}h_{ij}}{(q-h)^{2}}+ \dfrac{2h_{i}^{2}h_{j}}{(q-h)^{3}}\\
         \laplacian{G}&=\dfrac{2|h_{ij}|^{2}}{(q-h)^{2}}+\dfrac{2h_{i}h_{ijj}}{(q-h)^{2}}+\dfrac{8h_{i}h_{j}h_{ij}}{(q-h)^{3}}+\dfrac{2h_{i}^{2}h_{jj}}{(q-h)^{3}}+\dfrac{6h_{i}^{2}h_{j}^{2}}{(q-h)^{4}}.
    \end{split}
\end{equation*}
Hence,
\begin{equation*}
    \begin{split}
       \driftedlaplacian{}{\varphi}{G}&=\laplacian G-\langle\gradient{G},\gradient{\varphi}\rangle\\
       &=\dfrac{2|h_{ij}|^{2}}{(q-h)^{2}}+\dfrac{2h_{i}h_{ijj}}{(q-h)^{2}}+\dfrac{8h_{i}h_{j}h_{ij}}{(q-h)^{3}}+\dfrac{2h_{i}^{2}h_{jj}}{(q-h)^{3}}+\dfrac{6h_{i}^{2}h_{j}^{2}}{(q-h)^{4}}\\
       &\quad-\dfrac{2h_{i}\varphi_{j}h_{ij}}{(q-h)^{2}}- \dfrac{2h_{i}^{2}h_{j}\varphi_{j}}{(q-h)^{3}}.
    \end{split}
\end{equation*}
Substituting the Ricci identity $h_{ijj}=h_{jji}+\textup{Ric}_{ij}h_{j}$ into the previous equation, we obtain
\begin{equation*}
    \begin{split}
       \driftedlaplacian{}{\varphi}{G}&=\dfrac{2|h_{ij}|^{2}}{(q-h)^{2}}+\dfrac{2h_{i}(\driftedlaplacian{}{\varphi}{h})_{i}}{(q-h)^{2}}+\dfrac{2(\textup{Ric}_{ij}+\varphi_{ij})h_{i}h_{j}}{(q-h)^{2}}+\dfrac{8h_{i}h_{j}h_{ij}}{(q-h)^{3}}\\
       &\quad+\dfrac{2h_{i}^{2}(\driftedlaplacian{}{\varphi}{h})}{(q-h)^{3}}+\dfrac{6h_{i}^{2}h_{j}^{2}}{(q-h)^{4}}.
    \end{split}
\end{equation*}
Hence, from \eqref{thm-2:eq1}, we get
\begin{equation*}
    \begin{split}
       \driftedlaplacian{}{\varphi}{G}&=\dfrac{2|h_{ij}|^{2}}{(q-h)^{2}}+\dfrac{2h_{i}}{(q-h)^{2}}\left(-\frac{2}{p}h_{j}h_{ji}-p\mathcal{A}_{i}-\mathcal{B}e^{\frac{h}{p}(\varepsilon-1)}(\varepsilon-1)h_{i}\right)\\
       &\quad+\dfrac{2(\textup{Ric}_{ij}+\varphi_{ij})h_{i}h_{j}}{(q-h)^{2}}+\dfrac{8h_{i}h_{j}h_{ij}}{(q-h)^{3}}+\dfrac{6h_{i}^{2}h_{j}^{2}}{(q-h)^{4}}\\
       &\quad +\dfrac{2h_{i}^{2}}{(q-h)^{3}}\left(-\dfrac{1}{p}h_{i}^{2}-p\mathcal{A}-p\mathcal{B}e^{\frac{h}{p}(\varepsilon-1)}\right)\\
       &=\dfrac{2|h_{ij}|^{2}}{(q-h)^{2}}+\dfrac{2(\textup{Ric}_{ij}+\varphi_{ij})h_{i}h_{j}}{(q-h)^{2}}+\dfrac{8h_{i}h_{j}h_{ij}}{(q-h)^{3}}+\dfrac{6h_{i}^{2}h_{j}^{2}}{(q-h)^{4}}\\
       &\quad-\dfrac{4}{p}\dfrac{h_{i}h_{j}h_{ji}}{(q-h)^{2}}
       -\dfrac{2ph_{i}\mathcal{A}_{i}}{(q-h)^{2}}-\dfrac{2(\varepsilon-1)\mathcal{B}}{(q-h)^{2}}e^{\frac{h}{p}(\varepsilon-1)}h_{i}^{2}-\dfrac{2}{p}\dfrac{h_{i}^{4}}{(q-h)^{3}}\\
       &\quad-\dfrac{2p\mathcal{A}h_{i}^{2}}{(q-h)^{3}}-\dfrac{2p\mathcal{B}e^{\frac{h}{p}(\varepsilon-1)}h_{i}^{2}}{(q-h)^{3}}.
    \end{split}
\end{equation*}
Since $q-h\geqslant\delta>0$, we also have 
\begin{equation*}
    \dfrac{2|h_{ij}|^{2}}{(q-h)^{2}}+\dfrac{4h_{i}h_{j}h_{ij}}{(q-h)^{3}}+\dfrac{2h_{i}^{4}}{(q-h)^{4}}\geqslant0.
\end{equation*}
Then,
\begin{equation*}
    \begin{split}
       \driftedlaplacian{}{\varphi}{G}&\geqslant
      \dfrac{4h_{i}h_{j}h_{ij}}{(q-h)^{3}}+\dfrac{4h_{i}^{4}}{(q-h)^{4}}+\dfrac{2(\textup{Ric}_{ij}+\varphi_{ij})h_{i}h_{j}}{(q-h)^{2}}-\dfrac{4}{p}\dfrac{h_{i}h_{j}h_{ji}}{(q-h)^{2}}
       -\dfrac{2ph_{i}\mathcal{A}_{i}}{(q-h)^{2}}\\
       &\quad-\dfrac{2(\varepsilon-1)\mathcal{B}}{(q-h)^{2}}e^{\frac{h}{p}(\varepsilon-1)}h_{i}^{2}-\dfrac{2}{p}\dfrac{h_{i}^{4}}{(q-h)^{3}}-\dfrac{2p\mathcal{A}h_{i}^{2}}{(q-h)^{3}}-\dfrac{2p\mathcal{B}e^{\frac{h}{p}(\varepsilon-1)}h_{i}^{2}}{(q-h)^{3}}.
    \end{split}
\end{equation*}
The Ricci curvature condition $\textup{Ric}^{\varphi}\geqslant-(n-1)\textup{K}$ implies
\begin{equation*}
    (\textup{Ric}_{ij}+\varphi_{ij})h_{i}h_{j}\geqslant-(n-1)\textup{K}h_{i}^{2}.
\end{equation*}
On the other hand, from \eqref{thm-2:eq2}, we know that
\begin{equation*} \langle\gradient{h},\gradient{G}\rangle=h_{j}G_{j}=\dfrac{2h_{i}h_{j}h_{ij}}{(q-h)^{2}}+ \dfrac{2h_{i}^{4}}{(q-h)^{3}}.
\end{equation*}
Therefore,
\begin{equation*}
    \begin{split}
       \driftedlaplacian{}{\varphi}{G}&\geqslant
      \dfrac{2}{(q-h)}\langle\gradient{h},\gradient{G}\rangle-2(n-1)\textup{K}G-\dfrac{2}{p}\langle\gradient{h},\gradient{G}\rangle+2(q-h)G^{2}\\
       &\quad-\dfrac{2p}{(q-h)^{2}}\langle\gradient{h},\gradient{\mathcal{A}}\rangle-\dfrac{2p\mathcal{A}}{(q-h)}G-2\left(\varepsilon-1+\dfrac{p}{q-h}\right)\mathcal{B}e^{\frac{h}{p}(\varepsilon-1)}G\\
       &\geqslant\dfrac{2(p-q+h)}{p(q-h)}\langle\gradient{h},\gradient{G}\rangle-2(n-1)\textup{K}G+2(q-h)G^{2}\\
       &\quad-\dfrac{2p}{(q-h)^{2}}\langle\gradient{h},\gradient{\mathcal{A}}\rangle-\dfrac{2p\mathcal{A}}{(q-h)}G-2\left(\varepsilon-1+\dfrac{p}{q-h}\right)\mathcal{B}e^{\frac{h}{p}(\varepsilon-1)}G,
    \end{split}
\end{equation*}
which proves inequality \eqref{thm-2:eq3}.

Now, in order to prove the desired elliptic gradient estimate, we use a localization and cut-off function argument as well as an application of the maximum principle to a suitably defined localized function. First, consider $x_{0}\in B$, $R\geqslant2$ and denote by $r(x)=d(x,x_{0})$ the distance function from the minimizing geodesic emanating from $x_{0}$. Thus define:
\begin{equation}\label{thm-2:eq4}\phi:B\longrightarrow\reals,\quad \phi(x)=\overline{\phi}(r(x)),
\end{equation}
where $\text{supp}(\phi)\subset B(x_{0}, R)$ and $\overline{\phi}$ is a smooth cut-off function chosen to satisfy the following properties:
\begin{enumerate}[(a)]
    \item $\overline{\phi}:[0,+\infty)\longrightarrow\reals$, 
    $\text{supp}(\overline{\phi})\subset[0,R]$ 
    and $0\leqslant\overline{\phi}\leqslant1$;
    \item $\overline{\phi}=1$ in $[0,\frac{R}{2}]$ and 
    $\overline{\phi}^{\prime}=0$ in $[0,\frac{R}{2}]\cup[R,+\infty)$;
    \item $-C_{\nu}\dfrac{\overline{\phi}^{\nu}}{R}
    \leqslant \overline{\phi}^{\prime}\leqslant0$ 
    and $|\overline{\phi}^{\prime\prime}|\leqslant C_{\nu}
    \dfrac{\overline{\phi}^{\nu}}{R^{2}}$ for any $\nu\in (0,1)$ where $C_{\nu}>0$.
\end{enumerate}

The cut-off function \eqref{thm-2:eq4} was originally defined by Li and Yau~\cite{li1986parabolic} (see also Souplet and Zhang~\cite{souplet2006sharp}). 

To prove the desired estimate at every point $x\in B(x_{0},R/2)$ we consider the localized function $\phi G$ with $\phi$ as in \eqref{thm-2:eq4} and $G$ in \eqref{thm-2:eq2}. Hence, from \eqref{thm-2:eq3}, we have
\begin{equation*}
\begin{split}
&\dfrac{1}{2}\driftedlaplacian{}{\varphi}{(\phi G)}-
     \left[\dfrac{(p-q+h)}{p(q-h)}
     \gradient{h}+
     \dfrac{\gradient{\phi}}{\phi}\right]
     \gradient{(\phi G)}
     \\
     =&
     \dfrac{1}{2}\phi\driftedlaplacian{}{\varphi}{G}
     +
     \langle\gradient{\phi},\gradient{G}\rangle
     +
     \dfrac{1}{2}G\driftedlaplacian{}{\varphi}{\phi}
     -
     \left[\dfrac{(p-q+h)}{p(q-h)}
     \gradient{h}+
     \dfrac{\gradient{\phi}}{\phi}\right]
     (\phi\gradient{G}+G\gradient{\phi})\\
     =&
     \dfrac{1}{2}\phi\driftedlaplacian{}{\varphi}{G}
     -
     \dfrac{\phi(p-q+h)}{p(q-h)}\langle\gradient{h},\gradient{G}\rangle
     -
     \dfrac{(p-q+h)}{p(q-h)}\langle\gradient{h},\gradient{\phi}\rangle G
     +
     \dfrac{1}{2}G\driftedlaplacian{}{\varphi}{\phi}\\
     &-\dfrac{\norm{\gradient{\phi}}^{2}}{\phi}G.\\
     \geqslant&
     -(n-1)\textup{K}\phi G+(q-h)\phi G^{2}-\dfrac{p\phi}{(q-h)^{2}}\langle\gradient{h},\gradient{\mathcal{A}}\rangle-\dfrac{p\mathcal{A}}{(q-h)}\phi G\\
     &-\left(\varepsilon-1+\dfrac{p}{q-h}\right)\mathcal{B}e^{\frac{h}{p}(\varepsilon-1)}\phi G-
     \dfrac{(p-q+h)}{p(q-h)}\langle\gradient{h},\gradient{\phi}\rangle G
     +
     \dfrac{1}{2}G\driftedlaplacian{}{\varphi}{\phi}\\
     &-\dfrac{\norm{\gradient{\phi}}^{2}}{\phi}G.
\end{split}
\end{equation*}

Now, assume that $\phi G$ attains its maximum at the point $x_{1}\in B(x_{0},R)$. Since $r = r(x)$
is only Lipschitz continuous at the cut locus of $x_{0}$, we can apply an argument by Calabi~\cite{calabi1958extension} (see also Cheng and
Yau \cite{cheng1975differential}) to assume without loss of generality that $\phi$ is smooth at $x_{1}$. Consequently, the function defined by $\phi G$ is also smooth at $x_{1}$. 

Let us assume that $(\phi G)(x_{1})>0$. Otherwise, the gradient estimates will be a straightforward computation since $G(x)\leqslant0$
in $B(x_{0},R/2)$. Notice that at the maximum point $x_{1}$, we have the inequalities $\driftedlaplacian{}{\varphi}{(\phi G)}(x_{1})\leqslant0$ and
$\gradient{(\phi G)}(x_{1})=0$. Thus, at $x_{1}$, it is true that
\begin{equation}\label{thm-2:eq5}
\begin{split}
    (q-h)\phi G^{2}\leqslant&\underbrace{(n-1)\textup{K}\phi G}_{\textup{I}}+\underbrace{\dfrac{p-q+h}{p(q-h)}\langle\gradient{h},\gradient{\phi}\rangle G}_{\textup{II}}+\underbrace{\dfrac{\norm{\gradient{\phi}}^{2}}{\phi}G}_{\textup{III}}-\underbrace{\dfrac{1}{2}G\driftedlaplacian{}{\varphi}{\phi}}_{\textup{IV}}\\
    +&\underbrace{\dfrac{p\phi}{(q-h)^{2}}\langle\gradient{h},\gradient{\mathcal{A}}\rangle}_{\textup{V}}+\underbrace{\left(\varepsilon-1+\dfrac{p}{q-h}\right)\mathcal{B}u^{\varepsilon-1}\phi G}_{\textup{VI}} + \underbrace{\dfrac{p\mathcal{A}}{q-h}\phi G}_{\textup{VII}}.
\end{split}
\end{equation}

We now proceed by bounding each of the terms on the right side of \eqref{thm-2:eq5}. For this, we consider two cases: $r(x_{1})\leqslant1$ and $r(x_{1})\geqslant1$.

\textbf{Case 1:} $x_{1}\notin B(x_{0},1)$ i.e. $r(x_{1})\geqslant1$. During the calculations, we will frequently utilize the Cauchy–Schwarz and Young inequalities. In this case, we have the following estimates:

\vspace{0.1cm}
\noindent\textbf{Estimating $\textup{I}$}:
\begin{equation*}
\begin{split}
(n-1)\textup{K}\phi G&\leqslant\dfrac{\delta}{16}\phi G^{2}+c(\delta)\textup{K}^{2}.
\end{split}
\end{equation*}

\noindent \textbf{Estimating $\textup{II}$}:
\begin{align*}
\dfrac{p-q+h}{p(q-h)}\langle\gradient{h},\gradient{\phi}\rangle G \leqslant&
\dfrac{|p-q+h|}{p}\norm{\gradient{\phi}}G^{\frac{3}{2}}\\
=&\left[\frac{2}{3}\phi(q-h)G^{2}\right]^{\frac{3}{4}}\dfrac{\norm{\gradient{\phi}}|p-q+h|}{p\left[\frac{2}{3}\phi(q-h)\right]^{\frac{3}{4}}}\\
\leqslant&\dfrac{1}{2}\phi(q-h)G^{2}+c\left(\dfrac{\norm{\gradient{\phi}}}{\phi^{\frac{3}{4}}}\right)^{4}\frac{|p-q+h|^{4}}{p^{4}(q-h)^{3}}\\
\leqslant&\dfrac{1}{2}\phi(q-h)G^{2}+\frac{c|p-q+h|^{4}}{R^{4} p^{4}(q-h)^{3}}.
\end{align*}

\noindent\textbf{Estimating $\textup{III}$:}
\begin{equation*}
\dfrac{\norm{\gradient{\phi}}^{2}}{\phi}G \!=\! \phi^{\frac{1}{2}}G\dfrac{\norm{\gradient{\phi}}^{2}}{\phi^{\frac{3}{2}}} 
\leqslant\dfrac{\delta}{16}\phi G^{2}+c(\delta)\left(\dfrac{\norm{\gradient{\phi}}^{2}}{\phi^{\frac{3}{2}}}\right)^{2}
\leqslant\dfrac{\delta}{16}\phi G^{2}+\dfrac{c(\delta)}{R^{4}}.
\end{equation*}

From Wei-Wiley Laplacian comparison theorem (\cite[Theorem 3.1]{wei2009comparison}), we know that
\[\driftedlaplacian{}{\varphi}{r(x)}
\leqslant \gamma_{\driftedlaplacian{}{\varphi}{}}+
(n-1)\textup{K}(R-1),\]
where $\gamma_{\driftedlaplacian{}{\varphi}{}}:=
\max_{\partial B(x_{0},1)}\driftedlaplacian{}{\varphi}{r}$. Then we can estimate $\textup{IV}$ as follows.

\vspace{0.1cm}
\noindent\textbf{Estimating $\textup{IV}$:}
\begin{equation}\label{thm-2:eq6}
    \begin{split}
-\dfrac{1}{2}G\driftedlaplacian{}{\varphi}{\phi}
&=-\dfrac{1}{2}G\left[\overline{\phi}' \driftedlaplacian{}{\varphi}{r}+\overline{\phi}''
\norm{\gradient{r}}^{2}\right]
\\
&\leqslant-\dfrac{1}{2}G\left\{\overline{\phi}'[\gamma_{\driftedlaplacian{}{\varphi}{}}+
(n-1)\textup{K}(R-1)]+
\overline{\phi}''\right\}
\\
&\leqslant G\left\{|\overline{\phi}'|[\gamma_{\driftedlaplacian{}{\varphi}{}}+
(n-1)\textup{K}(R-1)]+
|\overline{\phi}''|\right\}
\\
&\leqslant \phi^{\frac{1}{2}}G\left\{\dfrac{|\overline{\phi}'|}{\phi^{\frac{1}{2}}}[\gamma_{\driftedlaplacian{}{\varphi}{}}+
(n-1)\textup{K}(R-1)]\right\}+
\phi^{\frac{1}{2}}G\dfrac{|\overline{\phi}''|}{\phi^{\frac{1}{2}}}
\\
&\leqslant\dfrac{\delta}{16}\phi G^{2}+c(\delta)\dfrac{|\overline{\phi}''|^{2}}{\phi}+c(\delta)[\gamma_{\driftedlaplacian{}{\varphi}{}}]^{2}\dfrac{|\overline{\phi}'|^{2}}{\phi}+c(n,\delta)\textup{K}^{2}(R-1)^{2}\dfrac{|\overline{\phi}'|^{2}}{\phi}\\
&\leqslant\dfrac{\delta}{16}\phi G^{2}+\dfrac{c(\delta)}{R^{4}}+\dfrac{c(\delta)[\gamma_{\driftedlaplacian{}{\varphi}{}}]^{2}}{R^{2}}+c(n,\delta)\textup{K}^{2}.
\end{split}
\end{equation}

\noindent\textbf{Estimating $\textup{V}$:}
\begin{equation*}
\dfrac{p\phi}{(q-h)^{2}}\langle\gradient{h},\gradient{\mathcal{A}}\rangle\!\!\leqslant\!\! \dfrac{p\phi}{(q-h)^{2}}\norm{\gradient{h}}\norm{\gradient{\mathcal{A}}}
\!\!=\!\!\dfrac{p\phi\norm{\gradient{\mathcal{A}}}G^{\frac{1}{2}}}{q-h}
\!\!\leqslant\!\!\dfrac{\delta}{16}\phi G^{2}+c(\delta)\!\!\!\!\!\sup_{B(x_{0},R)}\!\!\!\!\norm{\gradient{\mathcal{A}}}^{\frac{4}{3}}.
\end{equation*}

\noindent\textbf{Estimating $\textup{VI}$:}
\begin{align*}
&\left(\varepsilon-1+\dfrac{p}{q-h}\right)\mathcal{B}u^{\varepsilon-1}\phi G 
\leqslant \left[\left(\varepsilon-1+\dfrac{p}{q-h}\right)\mathcal{B}\right]^{+}u^{\varepsilon-1}\phi G\\
&\qquad\leqslant\dfrac{\delta}{16}\phi G^{2}+c(\delta)\left\{\left[\left(\varepsilon-1+\dfrac{p}{q-h}\right)\mathcal{B}\right]^{+}\right\}^{2}\sup_{B(x_{0},R)}u^{2(\varepsilon-1)}.
\end{align*}

\noindent\textbf{Estimating $\textup{VII}$:}
\begin{equation*}
\begin{split}
\dfrac{p\mathcal{A}}{q-h}\phi G&\leqslant \dfrac{\delta}{16}\phi G^{2}+c(p,\delta)(\mathcal{A}^{+})^{2}.
\end{split}
\end{equation*}
Combining \eqref{thm-2:eq5} with all the above estimates, we arrive at
\begin{align*}
\frac{1}{2}(q-h)\phi G^{2} & \leqslant \frac{c|p-q+h|^{4}}{R^{4} p^{4}(q-h)^{3}}+\dfrac{6\delta}{16}\phi G^{2}+\dfrac{c(\delta)}{R^{4}}+\dfrac{c(\delta)[\gamma_{\driftedlaplacian{}{\varphi}{}}]^{2}}{R^{2}}+c(n,\delta)\textup{K}^{2}\\
&\quad+c(p,\delta)(\mathcal{A}^{+})^{2}+c(\delta)\left\{\left[\left(\varepsilon-1+\dfrac{p}{q-h}\right)\mathcal{B}\right]^{+}\right\}^{2}\!\!\!\sup_{B(x_{0},R)}\!\!u^{2(\varepsilon-1)}\\
&\quad+c(\delta)\sup_{B(x_{0},R)}\norm{\gradient{\mathcal{A}}}^{\frac{4}{3}}.
\end{align*}
Since $q-h\geqslant \delta>0$, the previous estimate implies
\begin{align*}
    \phi G^{2}\leqslant& \frac{c|p-q+h|^{4}}{R^{4} p^{4}(q-h)^{4}}+\dfrac{c(\delta)}{R^{4}}+\dfrac{c(\delta)[\gamma_{\driftedlaplacian{}{\varphi}{}}]^{2}}{R^{2}}+c(n,\delta)\textup{K}^{2}+c(\delta)\sup_{B(x_{0},R)}\norm{\gradient{\mathcal{A}}}^{\frac{4}{3}}\\
&+c(p,\delta)(\mathcal{A}^{+})^{2}+c(\delta)\left\{\left[\left(\varepsilon-1+\dfrac{p}{q-h}\right)\mathcal{B}\right]^{+}\right\}^{2}\sup_{B(x_{0},R)}u^{2(\varepsilon-1)}\\
\leqslant& \dfrac{c(p,\delta)}{R^{4}}+\dfrac{c(\delta)[\gamma_{\driftedlaplacian{}{\varphi}{}}]^{2}}{R^{2}}+c(n,\delta)\textup{K}^{2}+c(\delta)\sup_{B(x_{0},R)}\norm{\gradient{\mathcal{A}}}^{\frac{4}{3}}\\
&+c(p,\delta)(\mathcal{A}^{+})^{2}+c(\delta)\left\{\left[\left(\varepsilon-1+\dfrac{p}{q-h}\right)\mathcal{B}\right]^{+}\right\}^{2}\sup_{B(x_{0},R)}u^{2(\varepsilon-1)},
\end{align*}
where we have used that
\begin{equation*}
   \frac{|p-q+h|}{p(q-h)}\leqslant\frac{|p|+|h-q|}{p(q-h)}=\frac{1}{q-h}+\frac{1}{p}\leqslant\frac{1}{\delta}+\frac{1}{p}.
\end{equation*}
So,
\begin{equation*}
\begin{split}
    (\phi^{2} G^{2})(x_{1})&\leqslant (\phi G^{2})(x_{1})\\
    &\leqslant\dfrac{c(p,\delta)}{R^{4}}+\dfrac{c(\delta)[\gamma_{\driftedlaplacian{}{\varphi}{}}]^{2}}{R^{2}}+c(n,\delta)\textup{K}^{2}+c(\delta)\sup_{B(x_{0},R)}\norm{\gradient{\mathcal{A}}}^{\frac{4}{3}}\\
&\quad+c(p,\delta)(\mathcal{A}^{+})^{2}+c(\delta)\left\{\left[\left(\varepsilon-1+\dfrac{p}{q-h}\right)\mathcal{B}\right]^{+}\right\}^{2}\sup_{B(x_{0},R)}u^{2(\varepsilon-1)}.
\end{split}
\end{equation*}
Since $\phi\equiv1$ in $r(x)<\frac{R}{2}$, it follows that
\begin{equation*}
\begin{split}
    G(x)&= (\phi G)(x)
    \leqslant (\phi G)(x_{1})\\
    &\leqslant\dfrac{c(p,\delta)}{R^{2}}+\dfrac{c(\delta)[\gamma_{\driftedlaplacian{}{\varphi}{}}]}{R}+c(n,\delta)\textup{K}+c(\delta)\sup_{B(x_{0},R)}\norm{\gradient{\mathcal{A}}}^{\frac{2}{3}}+c(p,\delta)\mathcal{A}^{+}\\
&\quad+c(\delta)\left[\left(\varepsilon-1+\dfrac{p}{q-h}\right)\mathcal{B}\right]^{+}\sup_{B(x_{0},R)}u^{(\varepsilon-1)}
\end{split}
\end{equation*}
that implies 
\begin{equation*}
\begin{split}
    \dfrac{\norm{\gradient{h}}}{q-h}(x)&\leqslant\dfrac{c(p,\delta)}{R}+\dfrac{c(\delta)\sqrt{\gamma_{\driftedlaplacian{}{\varphi}{}}}}{\sqrt{R}}+c(\delta,n)\sqrt{\textup{K}}+c(\delta)\sup_{B(x_{0},R)}\norm{\gradient{\mathcal{A}}}^{\frac{1}{3}}+c(p,\delta)\sqrt{\mathcal{A}^{+}}\\
    &\quad+c(\delta)\bigg{\{}\left[\left(\varepsilon-1+\dfrac{p}{q-h}\right)\mathcal{B}\right]^{+}\bigg{\}}^{\frac{1}{2}}\sup_{B(x_{0},R)}u^{\frac{(\varepsilon-1)}{2}},
\end{split}
\end{equation*}
for all $x\in B$ such that $r(x) < R/2$. Since $h = p\ln u$, substituting this into the above estimate completes the proof of theorem when $x_{1}\notin B(x_{0}, 1) \subset B(x_{0}, R)$ and $R\geqslant2$.

\textbf{Case 2:} The maximum point $x_{1}\in B(x_{0},1)$ i.e. $r(x_{1})<1$. In this case, for any $R\geqslant 2$, the function $\phi$ is constant on $B(x_{0}, R/2)$. Hence, from equation \eqref{thm-2:eq5}, we have
\begin{align*}
\nonumber    (q-h)\phi G^{2}\leqslant& \underbrace{(n-1)\textup{K}\phi G}_{\textup{I}}+\underbrace{\dfrac{p\phi}{(q-h)^{2}}\langle\gradient{h},\gradient{\mathcal{A}}\rangle}_{\textup{V}}\\
 \nonumber    &+\underbrace{\left(\varepsilon-1+\dfrac{p}{q-h}\right)\mathcal{B}u^{\varepsilon-1}\phi G}_{\textup{VI}}+\underbrace{\dfrac{p\mathcal{A}}{q-h}\phi G}_{\textup{VII}},\\
 \leqslant& \dfrac{\delta}{4}\phi G^{2}+c(\delta)\textup{K}^{2}+c(\delta)\sup_{B(x_{0},R)}\norm{\gradient{\mathcal{A}}}^{\frac{4}{3}}+c(p,\delta)(\mathcal{A}^{+})^{2}\\ \nonumber
&+c(\delta)\left\{\left[\left(\varepsilon-1+\dfrac{p}{q-h}\right)\mathcal{B}\right]^{+}\right\}^{2}\sup_{B(x_{0},R)}u^{2(\varepsilon-1)}.
\end{align*}
Since $q-h\geqslant \delta>0$, the previous estimate implies
\begin{align*}
\phi G^{2}\leqslant & \; c(\delta)\textup{K}^{2}+c(\delta)\sup_{B(x_{0},R)}\norm{\gradient{\mathcal{A}}}^{\frac{4}{3}}+c(p,\delta)(\mathcal{A}^{+})^{2}\\
&+c(\delta)\left\{\left[\left(\varepsilon-1+\dfrac{p}{q-h}\right)\mathcal{B}\right]^{+}\right\}^{2}\sup_{B(x_{0},R)}u^{2(\varepsilon-1)}.
\end{align*}
Recall that $\phi(x_{1})=1$ and $\phi(x)=1$ when $r(x)<R/2$. So,
\begin{align*}
    G(x)&= (\phi G)(x)
    \leqslant (\phi G)(x_{1})\\
    &\leqslant c(\delta)\textup{K}\!+\!c(\delta)\!\!\!\!\sup_{B(x_{0},R)}\!\!\!\!\norm{\gradient{\mathcal{A}}}^{\frac{2}{3}}\!\!+c(\delta)\!\!\left[\left(\varepsilon-1+\dfrac{p}{q-h}\right)\mathcal{B}\right]^{+}\!\!\!\!\! \sup_{B(x_{0},R)}\!\!\!\! u^{(\varepsilon-1)}\!\! + c(p,\delta)\mathcal{A}^{+}
\end{align*}
for all $x\in M$ such that $r(x)<R/2$. By the definition of $G(x)$, we prove that the estimate in the theorem still holds when $x_{1}\in B(x_{0},1)$.
\end{proof}

As an application of Theorem~\ref{Th2}, we prove Corollary~\ref{Cor1-triviality}.

\begin{proof} [Proof of Corollary~\ref{Cor1-triviality}] According to item \eqref{2b} of Theorem~\ref{thm:gEtM-c}, the function $u=f^{\frac{1}{\sigma(m)}}$ provides a solution of \eqref{eq:lichnerowicz-type-equation} on $B^n$. Then, if we fix a point ${x_{0}}\in B$ and apply Theorem~\ref{Th2} with parameters
\[\mathcal{A}(x)= \frac{\rho\scalarcurvature<g_{B}>(x)+\lambda(x)}{(\alpha-2m\rho)\sigma(m)},
\quad \mathcal{B}=-\frac{(\alpha-m\rho)R_{g_{F}}}{m(\alpha-2m\rho)\sigma(m)} \quad \hbox{and}\quad\varepsilon=1-2\sigma(m),\]
we get 
\begin{align*}
    \norm{\gradient{\ln u}}\leqslant & \; C \left[q-p\ln u\right]\Bigg{\{}\dfrac{1}{R}+\dfrac{\sqrt{\gamma_{\driftedlaplacian{}{w}{}}}}{\sqrt{R}}+o(R^{-\frac{1}{2}})\\
    &+\sqrt{\left[\left(\varepsilon-1+\dfrac{p}{q-p\ln u(x_{1})}\right)\mathcal{B}\right]^{+}}\sup_{B(x_{0},R)}u^{\frac{(\varepsilon-1)}{2}}\Bigg{\}}
\end{align*}
in $B(x_{0},R/2)$, where $x_{1}$ is a maximum point of $\phi G$  as mentioning in the proof of Theorem~\ref{Th2}. 

To prove item (a), we choose $p=1$ and $q=\frac{1}{\sigma(m)}+\ln D$ so that
\begin{equation*}
\begin{split}
  \varepsilon-1+\dfrac{p}{q-p\ln u(x_{1})}=-2\sigma(m)+\dfrac{1}{\frac{1}{\sigma(m)}+\ln\left(\frac{D}{u}\right)}\leqslant-\sigma(m)\leqslant0.
\end{split}
\end{equation*}
Since $\mathcal{B}\geqslant0$, one has
\[\left[\left(\varepsilon-1+\dfrac{p}{q-p\ln u(x_{1})}\right)\mathcal{B}\right]^{+}=0,\]
and then
\begin{equation*}
\begin{split}
    \norm{\gradient{\ln u(x_{0})}}&\leqslant C\left\{\frac{1}{\sigma(m)}+\ln\left[\frac{D}{u(x_{0})}\right]\right\}\Bigg{\{}\dfrac{1}{R}+\dfrac{\sqrt{\gamma_{\driftedlaplacian{}{w}{}}}}{\sqrt{R}}+o(R^{-\frac{1}{2}})\Bigg{\}}.
\end{split}
\end{equation*}
Now, since $\ln f(x)=o(r^{\frac{1}{2}}(x))$ near infinity, we also have $\ln u(x)=o(r^{\frac{1}{2}}(x))$ near infinity. Thus,
\begin{equation*}
\begin{split}
    \norm{\gradient{\ln u(x_{0})}}&\leqslant C\left[\frac{1}{\sigma(m)}+o(\sqrt{R})-\ln u(x_{0})\right]\Bigg{\{}\dfrac{1}{R}+\dfrac{\sqrt{\gamma_{\driftedlaplacian{}{w}{}}}}{\sqrt{R}}+o(R^{-\frac{1}{2}})\Bigg{\}}.
\end{split}
\end{equation*}
Letting $R\to+\infty$, we have $\norm{\gradient{\ln u(x_{0})}}=0$. Since $x_{0}$ is arbitrary, we get $u(x)=c\in\reals$ for all $x\in B$. Therefore, $f$ is constant.

To prove item (b), we consider $q=1+\ln\left[\frac{D}{u(x_{1})}\right]+p\ln u(x_{1})$. In this case, 
\begin{equation*}
\begin{split}
  \varepsilon-1+\dfrac{p}{q-p\ln u(x_{1})}&=-2\sigma(m)+\dfrac{p}{1+\ln\left[\frac{D}{u(x_{1})}\right]}.
\end{split}
\end{equation*}
Since 
\[\dfrac{1}{1+\ln\left[\frac{D}{u(x_{1})}\right]}>0,\]
we can choose a constant $p$ such that 
\begin{equation*}
\begin{split}
  \varepsilon-1+\dfrac{p}{q-p\ln u(x_{1})}&=-2\sigma(m)+\dfrac{p}{1+\ln\left[\frac{D}{u(x_{1})}\right]}\geqslant0.
\end{split}
\end{equation*}
Hence, from  $\mathcal{B}<0$, we obtain
\[\left[\left(\varepsilon-1+\dfrac{p}{q-p\ln u(x_{1})}\right)\mathcal{B}\right]^{+}=0,\]
and then
\begin{equation*}
\begin{split}
    \norm{\gradient{\ln u(x_{0})}}&\leqslant C\left\{1+\ln\left[\frac{D}{u(x_{0})}\right]\right\}\Bigg{\{}\dfrac{1}{R}+\dfrac{\sqrt{\gamma_{\driftedlaplacian{}{w}{}}}}{\sqrt{R}}+o(R^{-\frac{1}{2}})\Bigg{\}}.
\end{split}
\end{equation*}
The growth condition of $f$ implies
\begin{equation*}
\begin{split}
    \norm{\gradient{\ln u(x_{0})}}&\leqslant C\left[1+o(\sqrt{R})-\ln u(x_{0})\right]\Bigg{\{}\dfrac{1}{R}+\dfrac{\sqrt{\gamma_{\driftedlaplacian{}{w}{}}}}{\sqrt{R}}+o(R^{-\frac{1}{2}})\Bigg{\}}.
\end{split}
\end{equation*}
Letting $R\to+\infty$, we have $\norm{\gradient{\ln u(x_{0})}}=0$, which implies that $u$ is constant and consequently $f$ is constant.
\end{proof}

We pointed out that the assumptions on the constant $R_{g_F}(\alpha-m\rho)/(\alpha-2m\rho)$ of items~\eqref{3a} and \eqref{3b} of Corollary~\ref{Cor1-triviality} cannot be droped. Indeed, it is a consequence of our next result.

\begin{corollary}\label{Cor2-nonexistence}
It is not possible to construct gradient Einstein-type warped product $B^n\times_{f} F^m$ satisfying equation~\eqref{eq:lichnerowicz-type-equation} on $B^n$ with $\textup{Ric}_{g_B}^{w}\geqslant 0$ and coefficients satisfying either 
\begin{enumerate}[(a)]
\item\label{4a} $\sigma(m)>0$, \ $\dfrac{\rho\scalarcurvature<g_{B}>+\lambda}{\alpha-2m\rho}\leqslant 0$ $(=0)$ \ and \ $\dfrac{R_{g_F}(\alpha-m\rho)}{\alpha-2m\rho}>0$ $(\neq0)$,\quad\hbox{or}
\item\label{4b} $\sigma(m)<0$, \ $\dfrac{\rho\scalarcurvature<g_{B}>+\lambda}{\alpha-2m\rho}\geqslant 0$ $(=0)$ \ and \ $\dfrac{R_{g_F}(\alpha-m\rho)}{\alpha-2m\rho}<0$ $(\neq0)$
\end{enumerate}
both cases with asymptotic behavior 
$$\sup_{B\left(x_{0},R\right)}
\left\|\rho\gradient{\scalarcurvature<g_{B}>+\gradient{\lambda}}\right\|=o(R^{-\frac{3}{2}}) \;\; as \;\; R\to+\infty$$
and warping function satisfying $f(x)=e^{o(r^{\frac{1}{2}}(x))}$ near infinity.
\end{corollary}

\begin{proof}[Proof of Corollary~\ref{Cor2-nonexistence}]Suppose that there exists a gradient Einstein-type manifold satisfying the conditions of Corollary~\ref{Cor2-nonexistence}. Then $u=f^{\frac{1}{\sigma(m)}}$ provides a solution of \eqref{eq:lichnerowicz-type-equation} on $B^n$. Analogously to the proof of Corollary~\ref{Cor1-triviality}, we can show that $u$ is constant if either item~\eqref{4a} or item~\eqref{4b} holds. Substituting $u(x)\equiv c>0$ into \eqref{eq:lichnerowicz-type-equation} we get
\begin{equation*}
  \frac{\rho\scalarcurvature<g_{B}>+\lambda}{\alpha-2m\rho}
  =
  \frac{R_{g_{F}}(\alpha-m\rho)}{(\alpha-2m\rho)m}c^{-2\sigma(m)},
\end{equation*}
which is a contradiction. So, the Einstein-type metric cannot exist.
\end{proof}

\section{Further discussions}\label{Further-discussions}

Here, we provide applications of Corolaries~\ref{Cor1-triviality} and \ref{Cor2-nonexistence} on particular cases of gradient Einstein-type warped products $B^{n}\times_{f}F^{m}$, such as Ricci solitons, $\rho$-Einstein solitons and Einstein manifolds. In all cases, we continue with our standard notations $R_g$, $R_{g_B}$ and $R_{g_F}$ for the scalar curvature of $g=g_B+f^2g_F$, $g_B$ and $g_F$, respectively.  The reader will be able to observe that these corollaries provide applications for other solitons like gradient Yamabe solitons, quasi-Yamabe solitons, quasi-Einstein manifolds, gradient Ricci almost solitons and others. 

\subsection{Gradient Ricci soliton case} 
By using gradient estimates for the warping function, Borges~\cite[Corollary~1.5]{Valter} proved that any gradient expanding Ricci soliton $B^{n}\times_{f}F^{m}$ with noncompact base $B^{n}$, warping function $f\leqslant\sqrt{-R_{g_{F}}/m}$ and $R_{g_{F}}<0$ must be a standard Riemannian product. Our next result is a consequence of Corollary~\ref{Cor1-triviality} and complements Borges's result, enabling the analysis of gradient Ricci solitons with unbounded warping functions.

\begin{corollary}Let $B^n\times_{f} F^m$ be a gradient expanding Ricci soliton warped product with noncompact base, $f(x)=e^{o(r^{\frac{1}{2}}(x))}$ near infinity, $\textup{Ric}_{g_{B}}^{h}\geqslant 0$ and $R_{g_F}< 0$. Then, it must be a standard Riemannian product.
\end{corollary}

In~\cite{tokura2022nonexistence}, the second author in joint work with Adriano, Pina and Barboza proved that there is no gradient steady Ricci soliton $B^{n}\times_{f}F^{m}$ with $R_{g_{F}}<0$, potential function $h$ and $k$-Bakry-Emery Ricci tensor satisfying
\[
\riccitensor<g_B>+\hessian{h}-\frac{1}{k}dh\otimes dh\geqslant0.
\]
Here, by assuming the less restrictive condition $\textup{Ric}^{h}_{g_B}\geqslant0$ we prove the following result as a consequence of Corollary~\ref{Cor2-nonexistence}.

\begin{corollary}\label{sec:discussions-1}There is no gradient steady Ricci soliton warped product $B^{n}\times_{f}F^{m}$ with $f(x)=e^{o(r^{\frac{1}{2}}(x))}$ near infinity,  $\textup{Ric}^{h}_{g_B}\geqslant0$ and $R_{g_{F}}\neq0$.
\end{corollary}

Corollaries~\ref{sec:discussions-1} tell us that by starting with a gradient steady Ricci soliton $(B^n,g_B)$, that is, $\textup{Ric}^{h}_{g_B}=0$, it is impossible to utilize such manifold as the base of a steady gradient Ricci soliton warped product $ B^n\times_{f}F^m$  with potential $h\circ\pi_{B}$,  $f(x)=e^{o(r^{\frac{1}{2}}(x))}$ near infinity and $R_{g_{F}}\neq0$. 

In~\cite{gomes2021note}, the first author together with Marrocos and Ribeiro proved that it is not possible to construct a gradient expanding Ricci soliton $B^{n}\times_{f}F^{m}$ with $R_{g_F}\geqslant 0$ and warping function $f$ satisfying either
$$f\in L^{1}(B^{n},e^{-\psi}dv_{g_B})\;\;\mbox{with}\;\;f(x)=O(e^{ar(x)^{2-b}})\;\;\mbox{as}\;\;r(x)\to+\infty,$$
for some constants $a,b>0,$  or
$$f\in L^{p}(B^{n},e^{-\psi}dv_{g_B}),\;\; 1<p\leqslant+\infty,$$  
where $\psi=h-m\ln{f}.$

Our next corollary provides a nonexistence result in a different setting as studied in~\cite[Theorem 1.3]{gomes2021note}.

\begin{corollary}\label{sec:discussions-2}There is no gradient expanding Ricci soliton warped product $B^n\times_{f} F^m$ with $f(x)=e^{o(r^{\frac{1}{2}}(x))}$ near infinity, $\textup{Ric}_{g_{B}}^{h}\geqslant 0$ and $R_{g_{F}}>0$.
\end{corollary}

\subsection{Einstein case}
Here, we work on the class of Einstein warped products $B^n\times_{f} F^m$ with $\textup{Ric}_{g_{B}}\geqslant 0$ and $f(x)=e^{o(r(x))}$ near infinity.

\begin{corollary}\label{sec:discussions-3}
Let $B^n\times_{f} F^m$ be an Einstein warped product with $\textup{Ric}_{g_{B}}\geqslant 0$, $R_g<0$ and $f(x)=e^{o(r(x))}$ near infinity. Then, it must be a standard Riemannian product with $R_{g_{F}}<0$.
\end{corollary}

\begin{proof}[Proof of Corollary~\ref{sec:discussions-3}]For Einstein manifolds we have $h$ constant. So, estimate IV \eqref{thm-2:eq6} becomes (cf. \cite[Eq. 2.14]{souplet2006sharp})
\[
\begin{aligned}
-G(\Delta \phi) & \leqslant \frac{\delta}{8} \psi G^2+ \frac{c(\delta)}{R^4}+\frac{c(\delta)\textup{K}}{R^2}.
\end{aligned}
\]
Therefore, from the hypothesis of Corollary~\ref{sec:discussions-3} we have the following gradient estimate for $u=f^{m}$
\begin{equation*}
\begin{split}
    \norm{\gradient{\ln u(x_{0})}}&\leqslant C\left\{1+o(r(x))-\ln u(x_{0})\right\}\Bigg{(}\dfrac{1}{R}\Bigg{)}, \qquad R\geqslant2,
\end{split}
\end{equation*}
where $x_{0}$ is an arbitrary point in $B^n$. 
Letting $R\to+\infty$, we have $\norm{\gradient{\ln u(x_{0})}}=0$ and then $u$ is constant on $B^n$.  Substituting $u(x)=c>0$ into \eqref{eq:lichnerowicz-type-equation} we get
\begin{equation*}
  0\leqslant\frac{R_{g_{F}}}{m}c^{-\frac{2}{m}}=\lambda=\frac{R_{g}}{n+m}<0,
\end{equation*}
which is a contradiction.
\end{proof}

Rimoldi \cite[Theorem~11]{rimoldi2011remark} proved that there is no Einstein warped product with $R_g<0$, bounded warping function and  $R_{g_{F}}\geqslant0$. Here, as an immediate consequence of Corollary~\ref{sec:discussions-3}, we complement Rimoldi's result for Einstein manifolds with an unbounded warping function.

\begin{corollary}\label{sec:discussions-4}There is no Einstein warped product $B^n\times_{f} F^m$ with $\textup{Ric}_{g_{B}}\geqslant 0$, $R_g<0$,  $f(x)=e^{o(r(x))}$ near infinity and $R_{g_{F}}\geqslant0$.
\end{corollary}

\begin{remark}In Corollaries~\ref{sec:discussions-3} and \ref{sec:discussions-4} we relax the growth condition for the warping function $f$ in Corollary~\ref{Cor2-nonexistence} since in this case the operator $\Delta_{\varphi}$ becomes $\Delta$ and the term $\gamma_{\driftedlaplacian{}{\varphi}{}}$ in Theorem~\ref{Th2} can be removed. Besides, the growth condition $f(x)=e^{o(r(x))}$ in Corollary~\ref{sec:discussions-3} is sharp due to example $\mathbb{H}^{n}=\mathbb{R}\times_{\cosh(t)}\mathbb{H}^{n-1}$. Note that the growth of $e^t$ and $\cosh(t)$ are the same. In addition, the growth condition $f(x)=e^{o(r(x))}$ in Corollary~\ref{sec:discussions-4} is also sharp due to $\mathbb{H}^{n}=\mathbb{R}\times_{e^{t}} \mathbb{R}^{n}$. For other examples of Einstein warped products $\reals\times_{e^{t}}F^{m}$ with $R_g<0$ and Ricci flat fiber we refer to Besse's book~\cite[Subsection 9.118]{besse2007einstein}.
\end{remark}

For Ricci flat warped products $B^n\times_{f} F^m$, Case~\cite{case2010nonexistence} obtained a gradient estimate for the warping function $f$ and proved that there is no Ricci flat warped product with nonconstant warping function and $R_{g_{F}}\leqslant0$. Here, in the setting of Ricci flat warped products with $\textup{Ric}_{g_{B}}\geqslant 0$, we prove the following rigidity result.

\begin{corollary}
Let $B^n\times_{f} F^m$ be a Ricci flat warped product with $\textup{Ric}_{g_{B}}\geqslant 0$ and $f(x)=e^{o(r(x))}$ near infinity. Then, it must be a standard Riemannian product with $R_{g_{F}}=0$.
\end{corollary}

\subsection{Gradient $\rho$-Einstein case}
Recall that a gradient $\rho$-Einstein soliton $(M^{d},g)$ is called a gradient Einstein soliton if $\rho=1/2$, a gradient traceless Ricci soliton if $\rho=1/d$, and a gradient Schouten soliton if $\rho=1/2(d-1)$. As an application of Corollary~\ref{Cor2-nonexistence}, we deduce the following three results.

\begin{corollary}There is no gradient Einstein soliton warped product 
  $B^{n}\times_{f}F^{m}$, $m\geq3$, with potential function $h$, warping function $f(x)=e^{o(r^{\frac{1}{2}}(x))}$ near infinity and geometric properties
\[\textup{Ric}^{w}_{g_B}\geqslant0,\quad  R_{g_{B}}=\text{cte}\geqslant-2\lambda, \quad R_{g_{F}}<0,\quad \mbox{where} \quad w=-\frac{h}{m-1}.
\]
\end{corollary}

\begin{corollary}There is no gradient traceless Ricci soliton warped product 
  $B^{n}\times_{f}F^{m}$, $n\neq m$, $n\neq1$ with potential function $h$, warping function $f(x)=e^{o(r^{\frac{1}{2}}(x))}$ near infinity and geometric properties
\[\textup{Ric}^{w}_{g_B}\geqslant0,\quad  R_{g_{B}}=\text{cte}\leqslant-\lambda (n+m), \quad R_{g_{F}}>0,\quad \mbox{where} \quad w=\frac{h(n+m)}{n-m}.
\]
\end{corollary}

\begin{corollary}\label{sec:discussions-5}There is no gradient Schouten soliton warped product 
  $B^{n}\times_{f}F^{m}$, $n\neq1$, with potential function $h$, warping function $f(x)=e^{o(r^{\frac{1}{2}}(x))}$ near infinity and geometric properties
\[\textup{Ric}^{w}_{g_B}\geqslant0,\  R_{g_{B}}=\text{cte}\leqslant-2\lambda(n+m-1), \ R_{g_{F}}>0,\ \mbox{where} \ w=\frac{h(n+m-1)}{n-1}.
  \]
\end{corollary}


As it was observed by Borges~\cite{borges2024rigidity}, one notable example of gradient Schouten soliton related to the class of rigid solitons is constructed as follows. Consider $n \geq 3$, $m < n$ and $\lambda \in \mathbb{R}$. Take an Einstein manifold $\left(F^m, g_F\right)$ of scalar curvature
\[
R_{g_F}=\frac{2(n-1) m \lambda}{2(n-1)-m}.
\]
For any $(x, p) \in \mathbb{R}^{n-m} \times F^m$ defines 
\[
h(x, p)=\frac{1}{2}\left(\frac{R_{g_F}}{2(n-1)}+\lambda\right)\|x\|^2,
\]
where $\|x\|^2$ denotes the Euclidean norm. It follows that $\mathbb{R}^{n-m} \times F^m$ is a gradient Schouten soliton with metric $g=g_{Euc}+g_F$ and potential function $h$. In the case $\lambda>0$, this example shows that the condition $R_{g_{B}}=cte\leqslant-2\lambda(n+m-1)$ in Corollary~\ref{sec:discussions-5} is necessary.

\section*{Acknowledgements}
The first author has been partially supported by Conselho Nacional de Desenvolvimento Científico e Tecnológico (CNPq), Grant 310458/2021-8, and Fundação de Amparo à Pesquisa do Estado de São Paulo (FAPESP), Grant 2023/11126-7. 

\bibliography{ref/main}

\bibliographystyle{plain}

\end{document}

%% file: pre/pkgs.tex
\usepackage{amsmath,amssymb,amsthm}
\usepackage[english]{babel}
\usepackage{enumerate}
\hypersetup{
  colorlinks  = true,
  linkcolor   = blue,
  anchorcolor = red,
  citecolor   = blue,
  filecolor   = red,
  urlcolor    = red,
  pdfauthor   = author
}
\usepackage{mathtools}
\usepackage{xcolor}
\usepackage{xparse}

%% file: pre/cmds.tex
\DeclarePairedDelimiter\norm{\lVert}{\rVert}

\newcommand{\reals}{\mathbb{R}}

\newcommand{\functionspace}[1]{C^{\infty}\left({#1}\right)}

\newcommand{\tensorproduct}[2]{{#1}\otimes{#2}}

\NewDocumentCommand{\gradient}{m d<>}{
  \IfNoValueTF{#2}{
    \nabla{#1}
  }{
    \nabla_{#2}{#1}
  }
}

\NewDocumentCommand{\hessian}{m d<> d() d[]}{
  \IfNoValueTF{#2}{
    \nabla{d{#1}}
  }{
    \nabla_{#2}{d{#1}}
  }
  \IfValueTF{#3}{
    \IfValueTF{#4}{
      \left({#3},{#4}\right)
    }{
      \left({#3},\cdot\right)
    }
  }{
    \IfValueT{#4}{
      \left(\cdot,{#4}\right)
    }
  }
}

\NewDocumentCommand{\laplacian}{m d<>}{
  \IfNoValueTF{#2}{
    \Delta{#1}
  }{
    \left(\Delta{#1}\right)_{#2}
  }
}

\NewDocumentCommand{\driftedlaplacian}{m m d<>}{
  \IfNoValueTF{#3}{
    \Delta_{#2}{#1}
  }{
    \left(\Delta_{#2}{#1}\right)_{#3}
  }
}

\newcommand{\divergent}[1]{\textup{div}\left({#1}\right)}

\NewDocumentCommand{\riccitensor}{d<> d() d[]}{
  \IfValueTF{#1}{
    \textup{Ric}_{#1}
  }{
    \textup{Ric}
  }
  \IfValueTF{#2}{
    \IfValueTF{#3}{
      \left({#2},{#3}\right)
    }{
      \left({#2},\cdot\right)
    }
  }{
    \IfValueT{#3}{
      \left(\cdot,{#3}\right)
    }
  }
}

\NewDocumentCommand{\scalarcurvature}{d<>}{
  \IfNoValueTF{#1}{
    R
  }{
    R_{#1}
  }
}

\NewDocumentCommand{\bakryemerytensor}{m m d<> d() d[]}{
  \IfNoValueTF{#3}{
    \textup{Rc}^{{#1},{#2}}
  }{
    \textup{Rc}^{{#1},{#2}}_{#3}
  }
  \IfValueTF{#4}{
    \IfValueTF{#5}{
      \left({#4},{#5}\right)
    }{
      \left({#4},\cdot\right)
    }
  }{
    \IfValueT{#5}{
      \left(\cdot,{#5}\right)
    }
  }
}